\documentclass[12pt]{article}

\usepackage[utf8]{inputenc}
\usepackage[T1]{fontenc}

\usepackage[largesc,osf]{newtxtext}
\usepackage{cite,enumitem,graphicx,mathtools,slashed}
\usepackage[dvipsnames,usenames]{xcolor}
\usepackage[amsthm,nosymbolsc,vvarbb]{newtxmath}
\usepackage[cal=euler,frak=euler]{mathalpha}

\usepackage[margin=3cm]{geometry}

\usepackage[all]{xy}
\usepackage{url}

\title{Characterization of principal bundles:\\ the commutative case}

\author{William J. Ugalde\\[12pt]
{\small
Escuela de Matemática, Universidad de Costa Rica,
11501 San José, Costa Rica%
\thanks{Email: william.ugalde@ucr.ac.cr}}}

\DeclareMathOperator{\pr}{pr}       

\newcommand{\al}{\alpha}            
\newcommand{\bt}{\beta}             
\newcommand{\ga}{\gamma}            
\newcommand{\ka}{\kappa}            
\newcommand{\om}{\omega}            
\newcommand{\sg}{\sigma}            
\newcommand{\vf}{\varphi}           

\newcommand{\bN}{\mathbb{N}}        
\newcommand{\bR}{\mathbb{R}}        
\newcommand{\bT}{\mathbb{T}}        
\newcommand{\bZ}{\mathbb{Z}}        

\newcommand{\sA}{\mathcal{A}}       
\newcommand{\sB}{\mathcal{B}}       
\newcommand{\sF}{\mathcal{F}}       
\newcommand{\sO}{\mathcal{O}}       
\newcommand{\sU}{\mathcal{U}}       

\newcommand{\hookto}{\hookrightarrow} 
\renewcommand{\leq}{\leqslant}      
\newcommand{\less}{\setminus}       
\newcommand{\lt}{\triangleright}    
\newcommand{\nisom}{\nsimeq}        
\newcommand{\rt}{\triangleleft}     
\newcommand{\x}{\times}             
\renewcommand{\.}{\cdot}            
\renewcommand{\:}{\colon}           

\newcommand{\cupycup}{\cup\cdots\cup} 
\newcommand{\ovl}{\overline}        
\newcommand{\unl}{\underline}       
\newcommand{\wt}{\widetilde}        

\newcommand{\set}[1]{\{\,#1\,\}}    
\newcommand{\word}[1]{\quad\text{#1}\quad} 

\makeatletter
\renewcommand{\section}{\@startsection{section}{1}{\z@}%
							 {-3.25ex \@plus -1ex \@minus -.2ex}%
							 {1.5ex \@plus.2ex}%
							 {\normalfont\large\bfseries}}
\renewcommand{\subsection}{\@startsection{subsection}{2}{\z@}%
							 {-3.25ex \@plus -1ex \@minus -.2ex}%
							 {1.5ex \@plus .2ex}%
							 {\normalfont\normalsize\bfseries}}
\makeatother

\numberwithin{equation}{section}        

\theoremstyle{plain}
\newtheorem{theorem}{Theorem}[section]  
\newtheorem{proposition}[theorem]{Proposition} 
\newtheorem{lemma}[theorem]{Lemma}      
\newtheorem{corollary}[theorem]{Corollary} 

\theoremstyle{definition}
\newtheorem{definition}[theorem]{Definition} 
\newtheorem{example}[theorem]{Example}  
\newtheorem{exercise}{Exercise}[section] 

\theoremstyle{remark}
\newtheorem*{remark}{Remark}

\DeclareRobustCommand{\qef}{
  \ifmmode 
  \else \leavevmode\unskip\penalty9999 \hbox{}\nobreak\hfill
  \fi
  \quad\hbox{\qefsymbol}}
\newcommand{\qefsymbol}{$\ominus$} 

\newlength{\bibitemsep}
\newlength{\bibparskip}
\let\oldthebibliography\thebibliography
\renewcommand\thebibliography[1]{%
  \oldthebibliography{#1}%
  \setlength{\parskip}{\bibitemsep}%
  \setlength{\itemsep}{\bibparskip}%
}  

\begin{document}

\maketitle

\begin{abstract}
A review of the characterization of principal bundles, through the
different properties of the action of a group and its related
canonical and translation maps, is presented. The work is divided in
three stages: a topological group acting on a topological space, a
discrete group acting on a smooth manifold, and a Lie group acting on
a smooth manifold.
\end{abstract}



\section{Introduction} 
\label{sec:intro}

In \cite{Brzezinski94}, motivated by the search for a way to pass
from the concept of a principal bundle to a possible substitute for
such an structure in the noncommutative geometry realm (see also
\cite{BaumHajacMatthesSzymanski} and~\cite{Kassel}), and making
reference to \cite[Definition 2.2, p.~41]{Husemoller}, one reads:
``The notion of principal bundle is equivalent to the existence of a
continuous translation map''.

The first intention of this paper is to establish that such a
statement is true in the topological case, insofar as the quotient
space has a Hausdorff topology. This is covered in
Section~\ref{sec:topological-bundles}. The second aim is to determine,
whether in the case of smooth principal bundles (i.e., having a smooth
manifold as a total space with a Lie group acting on it), the
translation map enjoys some sort of smoothness. This objective
requires to study the possibility of a smooth structure on the
quotient manifold, as well as on the spaces involved in the definition
of the translation map.

Before addressing the situation of a smooth total space, we study
discrete group actions on the total space in
Section~\ref{sec:discrete-case}. In
Section~\ref{sec:Husemoller-principal-bundles} a Lie group is acting
on a smooth manifold. As an immediate consequence of the results
already reviewed, versions of \cite[Theorem~8.1, p.~46]{Husemoller}
and \cite[Proposition 1.7, p.~18]{Berlineetal}, concerning cross
sections and equivariant maps on associated bundles, are presented in
Section~\ref{sec:equivariants} in the three settings dealt with
herein: a group acting on a topological space, a discrete group acting
on a smooth manifold, and a Lie group acting on a smooth manifold.

These notes are comprise a review of already known material; its
originality, if any, comes from how the ideas are presented.
Inspiration is taken from the references \cite{Brzezinski94,
BaumHajacMatthesSzymanski, Kassel} in which noncommutative geometric
replacements for principal fiber bundles are proposed and contain the
motivation driving these notes: to establish possible bases for the
eventual study of generalizations to the noncommutative setting of a
(commutative) principal bundle.

In \cite{Brzezinski94} a translation map is defined as follows:
``Assume that one has a manifold with a free action of a Lie
group~$G$. Every two points on an orbit are then related by a unique
element of~$G$. A translation map assigns such an element of $G$ to
any two points on an orbit.'' The translation map is a map from
$X \x_G X := \set{(z,y) : z,y \in X,\ \sO_z = \sO_y} \subseteq X \x X$
to~$G$. This definition contrasts with \cite{Husemoller}, who broaches
principal bundles on topological spaces, the maps involved are
continuous, and the subspace $X \x_G X$ inherits the relative topology
from $X \x X$. In the smooth setting, it will be necessary to address
the question of what sort of smooth structure, if any, there may be on
$X \x_G X$. Another possibility is to restrict the above translation
map to each orbit $\sO$, and to obtain a collection of maps from
$\sO \x \sO \subseteq X \x X$ to $G$. In this case, it is simpler to
provide each product of orbits $\sO \x \sO$ with a smooth structure.
As pointed out before, we want to understand which smooth properties
of the translation map for smooth principal bundles replace the
continuity conditions of topological principal bundles. The relevant
manifolds will be Hausdorff, so a first item will be to review the
Hausdorff topology on the quotient space, and later on, the existence
of a smooth structure on~it.

A more precise statement of the question would be, in the case of a
smooth principal bundle, what smoothness property is satisfied by the
translation map?

A continuous and free action of a topological group $G$ on a
topological space $X$ is equivalent to the existence of a homeomorphic
canonical map $\ga \: X \x G \to X \x_G X : (x,g) \mapsto (x,x\.g)$,
and in turn, it is equivalent to having a
continuous translation map $\tau \: X \x_G X \to G$. Further
requesting that the action be proper, amounts to have a Hausdorff
topology on the quotient space $X/G$ (or that $X \x_G X$ be closed in
$X \x X$). The motivation is the following. For a free and continuous
action to produce a canonical map $\ga$ which is a homeomorphism, the
map $\ga$ must be closed. Proper maps are also closed maps, and having
a proper action implies that the canonical map is proper and hence
closed. A free and proper action is called \textit{principal}. On
trying to move from topological spaces to smooth manifolds,
interesting questions to be addressed are: is there a smooth structure
on $X \x_G X$? and will the canonical map be a diffeomorphism?

\medskip

In Section~\ref{sec:topological-bundles} we study the notions of fiber
and principal bundle from \cite{Husemoller} and understand their
relation with continuous translation functions. We focus on the
properties of the action of the group on the total space and the
topological properties of the base space. A word of caution is needed.
What is called a principal $G$-space in \cite{Husemoller} is called a
\textit{quasi-principal} $G$-space herein. A similar situation happens
in~\cite{BaumHajacMatthesSzymanski}. The reason is simple: we have
added the prefix `quasi' since from
Lemma~\ref{lm:proper-action-characterization}, a quasi-principal
$G$-space is just one property short of having a proper action, i.e.,
what is missing is that the quotient topology on $X/G$ is Hausdorff
(or that $X \x_G X$ is closed in $X \x X$).

The main result of Section~\ref{sec:topological-bundles} is that in a
(quasi-)principal fiber $G$-bundle with Hausdorff base, the action of
$G$ on~$X$ is free and proper (i.e., principal) and hence each fiber
is homeomorphic to~$G$. Thus, the existence of a continuous
translation function alone does not characterize a principal
$G$-bundle (in the more general sense), it must also have a Hausdorff
base $X/G$. To illustrate that dual requirement, we present the Milnor
construction from \cite{Milnor} (see also \cite{Husemoller}), which
gives a universal (quasi) principal bundle in the sense that for any
(quasi) principal numerable%
\footnote{%
A principal $G$-bundle $(E,p,B)$ is \textit{numerable} provided there
is an open cover $\{U_i\}_{i\in I}$ of $B$ with a (locally finite)
partition of unity $\{u_i\}_{i\in I}$ such that
$u_i^{-1}((0,1]) \subseteq U_i$ and the bundle restricted to each
$U_i$ is trivial \cite{Husemoller}.}
$G$-bundle $(X,\pi,B)$, there is a (unique up to homotopy) continuous
map $f \: B \to B_G$ such that the pull-back bundle
$(f^*(E_G),\pi',B)$ is homeomorphic to $(X,\pi,B)$.

\medskip

For the case of a discrete group $G$ acting on a smooth manifold~$X$,
the first goal is to ensure that there is a smooth manifold structure
on the quotient $X/G$. We study covering space actions and discrete
actions. A covering space action provides a smooth atlas on~$X/G$; if
the action is moreover discrete, the quotient is a Hausdorff space. In
particular, each orbit is a submanifold of~$X$.

Theorem~\ref{thm:covering-principal-is-principal-discrete-case} says
that every covering $G$-bundle $(X,p,B)$ with a smooth and discrete
action is a smooth principal fiber $G$-bundle. Its corollary shows
that in a discrete $G$-bundle $(X,p,B)$ with a smooth action, $X/G$
has the structure of a smooth manifold locally diffeomorphic to~$X$
with continuous translation map. When $G$ is discrete,
$\tau \: X \x_G X \to G$ will be trivially smooth for any smooth
structure on $X \x_G X \to G$. The last part of
Section~\ref{sec:discrete-case} deals with a smooth fiber bundle with
fiber~$F$ associated to a discrete principal $G$-bundle.

In \cite[p.~55]{DuistermaatKolk} one reads: ``having a proper and free
action of $G$ on $M$ is equivalent to saying that $M$ is a principal
fiber bundle with structure group $G$.'' This makes reference to a
$C^k$ action, which implies a $C^k$ manifold structure on the base
space $M/G$. It is straightforward to derive the equivalent result for
the smooth situation. This result is known as the Quotient Manifold
Theorem that can be read in several places with different approaches,
like \cite{LeeJeffrey} and~\cite{Lee}. It replaces the requirement of
the quotient space being Hausdorff in the topological setting.

\medskip

Section~\ref{sec:Husemoller-principal-bundles} takes the notion of
principal bundle on topological spaces, and brings it to the smooth
situation with a Lie group action, asking what sort of smoothness is
obtained for the translation map.
Corollary~\ref{cor:diffeomorphic-orbits} establishes this property: if
a Lie group $G$ acts properly, freely and smoothly on a smooth
manifold~$X$, then the restriction of the translation map to each
orbit, $\tau_x \: x\.G \to G : x\.g \mapsto \tau(x,x\.g) = g$, is
smooth. This coincides with the interpretation of the translation map,
as a collection of smooth maps
$(x\.G) \x (x\.G) \to G : (x',x'') \mapsto \tau(x',x'')$, for
$x \in X$. This holds in any (smooth) principal $G$-bundle
$(X,p,B,G)$, through appropriate identifications given by a
diffeomorphism $f \: X/G \to B$. By Theorem~\ref{thm:QM}, smooth
principal actions are enough to obtain a smooth principal $G$-bundle.
Observe that this property is easily adapted to the continuous
situation: the restriction of the translation map to each orbit is
continuous.

The Quotient Manifold Theorem establishes the required manifold
structure on~$X/G$. An important tool is Lemma~\ref{lm:adapted-charts}
which gives the existence of charts adapted to the action of~$G$, for
$G$ a Lie group acting smoothly, freely, and properly on a smooth
manifold. These are charts with domains $U$ for which each $G$-orbit
is either disjoint from~$U$ or whose intersection with $U$ is a single
slice.

\medskip

Theorem~\ref{thm:sections-equiv-maps} of
Section~\ref{sec:equivariants}, taken from \cite{Husemoller}, states
that cross sections of an associated topological bundle correspond
bijectively with continuous equivariant maps from the original total
space to the fiber. Theorems
\ref{thm:smooth-cross-sections-G-discrete}
and~\ref{thm:smooth-cross-sections-G-Lie-group} extend to the cases in
which the group $G$ is a discrete group or a Lie group, respectively,
and the sections and equivariant maps involved are smooth:
\cite[Proposition 1.7, p.~18]{Berlineetal}.


\section{Principal actions on topological spaces} 
\label{sec:trans-map-topol-spaces}

A \textit{group action} of a group $G$ on a set $X$ is a homomorphism%
\footnote{%
With $e$ the identity of the group $G$, $\al(e)(x) = x$, for all
$x \in X$; and $\al(gh) = \al(g)\al(h)$, for all $g,h \in G$.}
$\al$ from $G$ to the group of bijective maps from $X$ to $X$; for
each $g \in G$, $\al(g)$ is a bijection on~$X$. $G$ is said to act
on~$X$ and the set $X$ is called a \textit{$G$-space}.

Given $x \in X$, the \textit{orbit through~$x$} is the set
$\sO_x = \set{\al(g)(x) : g \in G}$, i.e., the set of all possible
images of $x$ under the bijections that $G$ determines on $X$
through~$\al$. The space $X$ is the disjoint union of its orbits.

An action $\al$ of $G$ on $X$ can equivalently be described as a
mapping $\al \: G \x X \to X$ given by $\al(g,x) := \al(g)(x)$, with
the properties $\al(e,x) = x$ and $\al(gh,x) = \al(g,\al(h,x))$. With
this notation, the action is referred to as a \textit{left} action,
and it is written $g\.x = \al(g,x)$ (so as to have $e\.x = x$ and
$(gh)\.x = g\.(h\.x)$). The orbits are written as $G\.x = \sO_x$.

It is also possible to have a \textit{right} action,%
\footnote{%
Some references write $g \lt x$ for a left action and $x \rt g$ for a
right action. We will refrain from that, limiting the notation to
$g\.x$ and $x\.g$.}
simply by redefining $\al \: X \x G \to X$ by
$\al(x,g) := \al(g^{-1})(x)$. The $g^{-1}$ is needed to obtain a right
action. We use left or right actions without further notice; they
should be understood from the context.

The action is called \textit{free}
if $x\.g = x$ for some $x$ in $X$ implies $g = e$ (equivalently, if $g
\neq e$ then $x\.g \neq x$, for all $x \in X$).%
%
\footnote{%
In \cite{Giordano} there is a good collection of actions that are free
or unfree. There the concern is placed on \textit{minimal} actions:
those for which each orbit is dense.}

\begin{remark}
If $X$ is a topological space and $G$ is a topological group, the
action $X \x G \to X$ is required to be continuous. In consequence,
for each $g \in G$, the map $X \to X : x \mapsto x\.g$ is a continuous
map. It immediately follows that each such map is a homeomorphism with
inverse $x \mapsto x\.g^{-1}$. Even further, these maps can be
restricted to act from one orbit to itself,
$\sO_x \to \sO_x : x \mapsto x\.g$, producing homeomorphisms on the
orbit. In the smooth setting $X$ is a smooth manifold, $G$ is a Lie
group, the action $X \x G \to X$ is required to be smooth; for each
$g \in G$, the map $X \to X : x \mapsto x\.g$ is smooth and each of
these maps is a diffeomorphism with inverse $x \mapsto x\.g^{-1}$,
producing a diffeomorphism on any given orbit.
\end{remark}

Denote%
\footnote{%
It is important to understand this set $X \x_G X$. If the action is
transitive, since $\sO_x = X$ for all $x \in X$, it follows that
$X \x_G X = X \x X$. If the action is trivial: $\al(g) = 1_X$ for
every $g \in G$, then $X \x_G X$ is the diagonal in $X \x X$. Also,
$\{x\} \x \sO_x \subset X \x_G X$ and
$\sO_x \x \{x\} \subset X \x_G X$, for any $x \in X$, and evidently,
$(x,y) \in X \x_G X$ if and only if $(y,x) \in X \x_G X$, meaning that
$X \x_G X$ is symmetric in $X \x X$ with respect to the diagonal.}
$$
X \x_G X := \set{(x,y) : x,y \in X, \ \sO_x = \sO_y} \subseteq X \x X.
$$
The \textit{canonical map} corresponding to the action $\al$ of $G$
on~$X$ is the surjective map
$$
\ga \: X \x G \to X \x_G X : (x,g) \mapsto (x, x\.g).
$$

We wish to use the canonical map to understand the action.

\begin{lemma} 
\label{lm:free-action}
The action is free if and only if the canonical map is injective (and
hence, bijective).
\end{lemma}

\begin{proof}
If $(x,x\.g) = \ga(x,g) = \ga(y,h) = (y,y\.h)$, then $x = y$ and
$x = x\.hg^{-1}$. By freeness of the action, $h = g$ and~so $\ga$ is
injective. Reciprocally, if $\ga$ is injective and $x\.g = x$ for some
$x \in X$, then $\ga(x,g) = (x,x\.g) = (x,x) = \ga(x,e)$ implies
$g = e$, and thus $\al$ is free.
\end{proof}

When working with topological spaces, the subspace
$X \x_G X \subseteq X \x X$ has the relative topology. It is then
valid to ask about the continuity of the canonical map.

\begin{lemma} 
\label{lm:canonical-map-continuity}
The action is continuous if and only if the canonical map is
continuous.
\end{lemma}

\begin{proof}
Consider the action as a map $\al \: X  \x G \to X$ with
$\al = \pr_2 \circ \ga :X \x G \to X \x_G X \to X$, where $\pr_2$ is
the projection on the second factor. It is a continuous map since
$X \x_G X$ has the relative topology, and this projection is actually
the restriction to $X \x_G X$ of the continuous projection
$X \x X \to X$. Thus, the continuity of $\ga$ implies the continuity
of~$\al$.

Since the canonical map is the restriction (on the range) to 
$X \x_G X$ of the map 
\[
\wt{\ga} : X \x G \to X \x X : (x, g) \mapsto (x, x\.g),
\]
it is enough to verify the continuity of $\wt{\ga}$ to conclude that
of~$\ga$. Evidently the projection map $X \x G \to X$ is continuous.
If the action $\al \: X \x G \to X : (x,g) \mapsto x\.g$ is
continuous, and the projection
$\pr_1 \: X \x G \to X : (x,g) \mapsto x$ is clearly continuous, then
$\wt{\ga}$ is the product of these two continuous maps, making it a
continuous map, too.
\end{proof}

As a result of the previous two lemmas, the action is free and
continuous if and only if the canonical map is bijective (injective)
and continuous.

Now we place ourselves in the setting of an invertible canonical map,
i.e., that of a free action. Its inverse generates an important map
associated with a free action of a group $G$ on a space $X$, namely
the \textit{translation map}, as follows.%
\footnote{In \cite[pp.~255--256]{Bourbaki} the names `canonical map'
and `translation map' are interchanged in the sense that
the map $\vf \: C \to G$ that assigns to $(x,y) \in C$ the
unique element $g\in G$ with $g\.x = y$ is therein called the
\textit{canonical mapping}, where $C$ (our $X \x_G X$) is the graph of
the relation determined on $X$ by the action of~$G$. }
Consider the composition map
$$
X  \x_G X \stackrel{\ga^{-1}}{\longrightarrow} X  \x G
\stackrel{\pr_2}{\longrightarrow} G
: (x,y) = (x,x\.g) \mapsto (x,g) \mapsto g.
$$

Let $x,y \in X$ belong to the same orbit. There is a unique group
element $g \in G$ such that $y = x\.g$. The map
$\tau \: X \x_G X \to G$, that assigns to the pair $(x,x\.g)$ the
element $g$ is called the \textit{translation map}. Evidently,
$\tau(x,y) = \tau(y,x)^{-1}$, $\tau(x,x) = e$, and
$\tau(x,y) \tau(y,z) = \tau(x,z)$.

\begin{remark}
It is possible to restrict the translation map to  
$\tau|_x \: \sO_x \x \sO_x \to G$, for every $x \in X$. Each set
$\sO_x \x \sO_x$ inherits the structure from the orbit, such as a
topology or smooth structure. See the remark after
Lemma~\ref{lm:proper-action-characterization} and
Corollary~\ref{cor:diffeomorphic-orbits}.
\end{remark}

\begin{lemma} 
\label{lm:XxG-homeomorphism}
The translation map of a continuous action is continuous if and only
if the canonical map has a continuous inverse, i.e., it is a
homeomorphism, $X \x G \approx X \x_G X$.
\end{lemma}

\begin{proof}
Since $\tau = \pr_2 \circ \ga^{-1}$, the continuity of $\ga^{-1}$
implies that of~$\tau$. Reciprocally, assume that
$\tau \: X \x_G X \to G$ is continuous. Since
$\ga^{-1}(x,x\.g) = (\pr_1(x,x\.g), \tau(x,x\.g))$, $\ga^{-1}$ is the
product of the two continuous maps, $\pr_1 \: X \x_G X \to X$ and the
translation map.
\end{proof}

Thus, in the setting of free and continuous actions, the continuity of
the translation map identifies the space $X \x_G X$ with $X \x G$
homeomorphically.

When working with smooth manifolds an important question to ask is
whether there is a smooth structure on such a set $X \x_G X$?
Furthermore, is the canonical map a diffeomorphism? Is it equivalent
to requiring that the translation map be a smooth map? The answer to
these questions is dealt with in
Section~\ref{sec:Husemoller-principal-bundles}.

Next, we study free actions for which the canonical map is a
homeomorphism, and hence, the translation map is continuous. For the
canonical map of a free and continuous action to be a homeomorphism,
we lack the continuity of its inverse $\ga^{-1}$, equivalently, we are
missing that the canonical map $\ga$ be a closed map.
Corollary~\ref{cor:continuous-free-proper-action} states that if the
action is proper and free, then the canonical map is indeed a
homeomorphism.

A map $f \: X \to Y$ is \textit{proper} if and only if it is
continuous and $f \x 1_T \: X \x T \to Y \x T$ is a closed map, for
every topological space~$T$.

\begin{remark}
By taking $T$ as a set with one element, it follows that any proper
map is a closed map. Also, the composition of proper maps is
immediately a proper map. Furthermore, proper maps characterize
compact spaces $X$ in the sense that, if the map
$\bar p \: X \to \{p\}$ into a singleton is proper, then $X$ is
compact. Indeed, following \cite{Bourbaki} we show that every filter
in $X$ has a cluster point. Let $\sF$ be a filter on $X$ and for
$\om \notin X$ consider the filter on $X' = X \cup \{\om\}$ given by
$\sF' = \set{F \cup \{\om\} : F \in \sF}$. Note that $\om$ belongs to
the closure of $X$ in~$X'$ ($X$ is not closed in $X'$ as $\{\om\}$ is
not open in~$X'$). In $X \x X'$ consider the closed set $\ovl{D}$
given by the closure of the set
$D = \set{(x,x) : x \in X} \subseteq X \x X'$. By hypothesis, the
projection of $\ovl{D}$ on $X'$ is closed. Since $X$ is contained in
this projection we know that $\om$ belongs to this projection, i.e.,
there is an $x \in X$ with $(x,\om) \in \ovl{D}$. If $V$ is a
neighborhood of~$x$ in $X$ and $F \in \sF$, then
$\emptyset \neq (V \x F) \cap D$, implying that
$V \cap F \neq \emptyset$. Hence, $x$ is a cluster point for~$\sF$,
and $X$ is compact.
\end{remark}

The following characterization 
\cite[Proposition~7, p.~104]{Bourbaki} is suitable for our purposes.

\begin{lemma} 
\label{lm:proper-map}
If $X$ is Hausdorff and $Y$ is locally compact, a continuous map
$f \: X \to Y$ is proper if and only if the inverse image of a compact
set is compact.
\end{lemma}

\begin{proof}
Observe that if $f \: X \to Y$ is a proper map and $S \subseteq Y$ is
any subset, then the map $f|_S \: f^{-1}(S) \to S$ is a proper map.
Indeed, if $T$ is any topological space and 
$C \subseteq f^{-1}(S) \x T$ is a closed subset, then 
$(f|_S \x 1_T)(C) = (f \x 1_T)(C) \cap (S \x T)$, which is closed
in $S \x T$. In this way, if $K \subseteq Y$ is a compact set, then
the composition of maps
$\bar p \circ f|_S \: f^{-1}(K) \to K \to \{p\}$ is a proper map,
meaning that $f^{-1}(K)$ is proper in~$X$, thanks to the previous
remark.

For the converse, let $T$ be a topological space and
consider a closed set $C \subseteq X \x T$. We must show that
$(f \x 1_T)(C)$ is closed in $Y \x T$. The complement of
$(f \x 1_T)(C)$ is not open in $Y \x T$ if and only if every open
neighborhood of some point $(y,t) \notin (f \x 1_T)(C)$ intersects
$(f \x 1_T)(C)$. This is equivalent to the existence of a net
$((y_j,t_j)) \subseteq (f \x 1_T)(C)$ with
$(f(x_j),t_j) = (y_j,t_j) \to (y,t)$, for some net
$((x_j,t_j)) \subseteq C$. The set $\set{y,y_j : j \in J}$ is compact
in~$Y$, so its inverse image under $f$ is compact in~$X$. In this way,
there is $x \in X$ and a subnet $(x_{j_k})$ such that
$(x_{j_k},t_{j_k}) \to (x,t)$, meaning that $(x,t) \in C$. By
continuity, $(y,t) = (f(x),t) \in (f \x 1_T)(C)$, contrary to
hypothesis. Therefore, $f$ is a proper map.%
\footnote{%
In the spaces are metrizable, like manifolds, it is enough to work
with sequences instead of nets.}
\end{proof}

A continuous action of $G$ on $X$ is by definition a \textit{proper
action} if and only if the extension of its canonical map,
$\wt{\ga} \: X \x G \to X \x X : (z,g) \mapsto (z,z\.g)$ is a proper
map \cite[III-4, Definition~1]{Bourbaki}. In particular, $\wt{\ga}$ is
a closed map. Since $X \x_G X$ has the relative topology, it follows
that the canonical map itself is also a continuous and closed map,
precisely the property we needed. Again and more generally, since
$X \x_G X$ has the relative topology, it follows that the canonical
map itself is a proper map.

The action $\al$ of $G$ on $X$ defines on $X$ an equivalence relation:
$y \sim x$ if and only if there is a $g \in G$ with $y = x\.g$. The
quotient space by this relation is $X/G$ and we denote
$\pi \: X \to X/G$ the surjective map $x \mapsto x\.G$. The set
$X \x_G X$ is not $X/G$, but rather the graph of the equivalence
relation that the action $\al$ defines on~$X$. The following statement
is from \cite{DuistermaatKolk}. We write the proof using our notation.

\begin{lemma} 
\label{lm:DuKo}
The quotient topology of $X/G$ is Hausdorff if and only if $X \x_G X$
is closed in $X \x X$.
\end{lemma}

\begin{proof}
Let $(x,y) \notin X \x_G X$. Since $x \nsim y$,
there are disjoint open sets $U_G$ and $V_G$ in
$X/G$, each containing the respective classes of~$x$ and of~$y$.
Consider the open sets in $X$, $U := \pi^{-1}(U_G)$ and
$V := \pi^{-1}(V_G)$, with $(x,y) \in U \x V$. If
$(z,w) \in (U \x V) \cap (X \x_G X)$, it would follow that
$\pi(z) \in U_G$, $\pi(w) \in V_G$, and $w \sim z$; whereby
$\pi(z) = \pi(w) \in U_G \cap V_G$, which is impossible. Thus,
$X \x_G X$ is a closed set in $X \x X$.

Before verifying the reciprocal statement, let $U$ be an open set
in~$X$. Now, $\pi(U)$ is open in $X/G$ if and only if
$\pi^{-1}(\pi(U)) = \set{y \in X : y \sim x \text{ for some } x \in U}
= \bigcup_{g\in G} \al(g)(U)$ is open in~$X$. This is true since each
$\al(g) \: X \to X$ is a homeomorphism.

Assume $X \x_G X$ is a closed subset of $X \x X$. Let $x\.G$, $y\.G$
be two different equivalence classes in $X/G$, with corresponding
representatives $x,y$ in~$X$, so that $(x,y) \notin X \x_G X$. Let 
$U$, $V$ be open sets in $X$ with $(x,y) \in U \x V$ and
$(U \x V) \cap (X \x_G X) = \emptyset$. Consider the open sets
$\pi(U)$ and $\pi(V)$ in $X/G$. Evidently, $x\.G \in \pi(U)$ and
$y\.G \in \pi(V)$. Lastly, if $\pi(U) \cap \pi(V) \neq \emptyset$,
then there would be $z \in X$ with $z\.G \in \pi(U) \cap \pi(V)$. That
is, there would be $g,h \in G$ with
$(z\.g,z\.h) = (z\.g,(z\.g)\.(g^{-1}h)) \in (U \x V) \cap (X \x_G X)$,
which is false. Therefore, the quotient topology of $X/G$ is
Hausdorff.
\end{proof}

Next, we characterize \textit{free and proper actions} in terms of the
Hausdorff topology on the quotient space and the continuity of the
translation map. Inspiration is taken from \cite{DuistermaatKolk} and
\cite[Proposition~6, p.~255]{Bourbaki}. The next result may be
compared with those sources.

\begin{lemma} 
\label{lm:proper-action-characterization}
A continuous and free action of a topological group~$G$ on a
topological space~$X$ is proper if and only if the quotient space
$X/G$ has a Hausdorff topology (i.e., $X \x_G X$ is closed in
$X \x X$) and the translation map $\tau$ is continuous (i.e., the
canonical map $\ga$ is a homeomorphism).
\end{lemma}

\begin{proof}
Let $\al$ be a free, continuous, and proper action of~$G$ on~$X$.

In particular, the map $\wt{\ga}: X \x G \to X \x X$
and the canonical map $\ga$ itself are proper maps. For every
one-point space $\{p\}$, the map
$\wt{\ga} \x 1_{\{p\}} \: X \x G \x \{p\} \to X \x X \x \{p\}$
is a closed map. Using a simple homeomorphism, it follows that
$\wt{\ga}$ is a closed map. Thus, $X \x_G X = \wt{\ga}(X \x G)$ is
closed in $X \x X$, and so the quotient topology on $X/G$ is
Hausdorff.

To show that the translation map $\tau$ is continuous is equivalent to
showing that $\ga$ is a homeomorphism. From the hypothesis we already
know that the canonical map $\ga$ is continuous, invertible, and
closed.

Lastly, assume that the action of $G$ on $X$ is free, continuous, and
that its canonical map $\ga$ is a homeomorphism of $X \x G$ onto the
closed set $X \x_G X$ in $X \x X$. Then $\wt{\ga}$ is continuous from
$X \x G \to X \x X$, since $\wt{\ga}(z,g) = (z,z\.g)$ is a composition
of continuous maps.

It immediately follows that the map
$\ga \x 1_T \: X \x G \x T \to X \x_G X \x T$ is a homeomorphism,
for any topological space $T$. Let $C$ be a closed set in
$X \x G \x T$. To finish, we must show that
$(\wt{\ga} \x 1_T)(C)$ is closed in $X \x X \x T$. Since $\ga$ is a
homeomorphism and $\ga(X \x G) = X \x_G X$, then
$(\wt{\ga} \x 1_T)(C) = (\ga \x 1_T)(C)$ is the image of a closed
set under a homeomorphism, and thereby is closed. Therefore the
extension $\wt{\ga}$ of the canonical map is a proper map.
\end{proof}

\begin{remark}
If $\tau$ is continuous, then its restriction to each orbit, 
$\sO_x \to G : x' \mapsto \tau(x,x')$ is also continuous. Moreover,
the maps $\sO_x \x \sO_x \to G : (x',x'') \mapsto \tau(x',x'')$ are
continuous. If the action is proper and free, then $X/G$ is Hausdorff
and each of these maps is continuous. This second property carries
through to the smooth situation.
Corollary~\ref{cor:diffeomorphic-orbits} shows that in the smooth
setting, the restriction of the translation map to each orbit,
$\sO_x \to G : x' \mapsto \tau(x,x')$, is smooth. These results (in
the topological and smooth situations) are indicators that, when
seeking an algebraic characterization to pass from (commutative)
principal bundles to noncommutative principal bundles, it is not the
translation map from $X \x_G X$ to $G$, but the restricted maps from
the products of each orbit with itself to the group, that form the
appropriate key to open this door. As regards the canonical map
itself, this ``fiber-wise'' characterization is equivalent to the
continuity of the restriction of its inverse map
$\ga_x^{-1} \: \sO \x \sO \to \sO \x G : (x,x\.g) \mapsto (x,g)$.
\end{remark}

\begin{remark}
In \cite{Biller} an action is called \textit{Cartan} is every stabilizer 
$G_x = \set{g \in G : g\.x = x}$ is compact (called quasi-compact
there) and for each $x \in X$ and every neighborhood $U \subseteq G$
of~$G_x$, there is a neighborhood $V \subseteq X$ of~$x$ such that
$\sg_G(V,V) = \set{g \in G : g\.V \cap V \neq \emptyset} \subseteq U$.
In the case of a free action, every stabilizer is $\{e\}$ and the
previous property becomes
\begin{equation}
\label{eq:Cartan-in-Biller} 
\forall\,x \in X, \ \forall\,U \in \sU(e), \ \exists\, V \in \sU(x)
: \sg_G(V,V) \subseteq U.
\end{equation}
In that reference, it is shown that an action of a topological
group~$G$ on a topological space~$X$ is proper iff the quotient space
$X/G$ has a Hausdorff topology and it is a \textit{Cartan action}.

In \cite{Biller2} it is proved that the action of $G$ on $X$ is proper
if and only if all stabilizers are compact (there again called
quasi-compact) and for all $x,x' \in X$ and every neighborhood
$U \subseteq G$ of $\sg_G(x,x') = \set{g \in G : g\.x = x'}$, there
are two neighborhoods $V,V' \subseteq X$ with $V \in \sU(x)$, 
$V' \in \sU(x')$ and $\sg_G(V,V') \subseteq U$. If the action is free,
then $\sg_G(x,x') = \{\tau(x,x')\}$ and the previous relation is
written~as:
\begin{equation}
\label{eq:Cartan-in-Biller2} 
\forall\, x,x' \in X, \ \forall \, U \in \sU(\tau(x,x')), \
\exists\, V \in \sU(x), \ \exists\, V' \in \sU(x')
: \sg_G(V,V') \subseteq U,
\end{equation}
with \eqref{eq:Cartan-in-Biller} equivalent to
\eqref{eq:Cartan-in-Biller2} and equivalent to the continuity of the
translation map. As a result of these relations and
Lemma~\ref{lm:proper-action-characterization}, an action is principal
if and only if the quotient space $X/G$ is Hausdorff and
\eqref{eq:Cartan-in-Biller} or~\eqref{eq:Cartan-in-Biller2} holds.

In \cite[Proposition 1.2]{Ellwood} it is shown that if $G$ is locally
compact and $X$ is locally compact and Hausdorff, then the action is
principal if and only if the canonical map is a closed embedding.
\end{remark}

\begin{corollary} 
\label{cor:continuous-free-proper-action}
If the (continuous) action of $G$ on $X$ is free and proper, its
canonical map is a homeomorphism from $X \x G$ onto the closed set
$X \x_G X$ in $X \x X$.
\end{corollary}

An action that is proper and free is called a
\textit{principal action}.%
\footnote{%
In \cite{BaumHajacMatthesSzymanski} (mostly based on the general
reference \cite{Bourbaki}) an action is called \textit{principal} if
it is free and the map $X \x G \to X \x X$ given by 
$(x,g) \mapsto (x, x\.g)$ is a proper map. The properness of the
action in this case guarantees the continuity of the translation map.
This notion (see \cite[Proposition~6, p.~255]{Bourbaki}) is equivalent
with the notion of H.~Cartan \cite[p.~6-05]{Cartan}, who requires the
continuity of the translation map and that $X \x_{X/G} X$ be closed in
$X \x X$.}
In Section~\ref{sec:topological-bundles} we explore such actions.

In \cite{Ellwood} one reads:
``When the action of $G$ is both free and proper the quotient space
$X/G$ is Hausdorff and the quotient map $\pi_{X/G} \: X \to X/G$
determines a principal $G$-bundle over $X/G$. Of course, it is this
property that motivates the terminology \textit{principal group
action} and explains their essential role in mathematics and physics.
Outside of the principal context many things can go wrong. When $G$
acts properly but not freely, we retain a Hausdorff quotient but lose
the canonical fiber, i.e., $\pi_{X/G}^{-1}([x]_G) \nisom G$ for some
$x \in X$. On the other hand when $G$ acts freely but not properly,
the quotient topology may be non-Hausdorff and quite useless in
general.''


\section{Topological principal bundles} 
\label{sec:topological-bundles}

In this section we review the notions of fiber bundle and principal
bundle from \cite{Husemoller} and understand their relation with
continuous translation functions.

In \cite{Husemoller}, as a starting point, a \textit{bundle} is
defined as a triple $(E,p,B)$ with $E$ and $B$ topological spaces and
$p \: E \to B$ a continuous and surjective function. Each $p^{-1}(b)$
is called the fiber of the bundle over $b \in B$. Up to this point,
not much structure is required in such an object. Even further, the
different fibers are not required to be isomorphic to each other in
any sense.

Later on in \cite{Husemoller}, for each $G$-space $X$ (with $G$ a
topological group) one forms the bundle $(X,\pi,X/G)$, where $X/G$ is
the set of equivalence classes given by the relation $y \sim x$ if and
only if $x\.g = y$ for some $g \in G$, and by considering on $X/G$ the
quotient topology. Then by definition, any other bundle $(X,p,B)$ with
total space $X$ is called a \textit{$G$-bundle} if and only if
$(X,\pi,X/G)$ and $(X,p,B)$ are \textit{isomorphic} for some $G$-space
structure on $X$ by an isomorphism
$(1,f) \: (X,\pi,X/G) \to (X,p,B)$, where $f \: X/G \to B$ is a
\textit{homeomorphism}. In other words, $p = f \circ \pi$.

As a result, if $(X,p,B)$ is a $G$-bundle, then each fiber over $B$
corresponds to an orbit of the action of $G$ on~$X$. Indeed, if
$b \in B$ and $\pi(x) \in X/G$ with $f(\pi(x)) = b$ for some 
$x \in X$, take $y \in p^{-1}(b)$; then
$f(\pi(x)) = b = p(y) = f(\pi(y))$. So $\pi(x) = \pi(y)$; it follows
that $p^{-1}(b) \subseteq \sO_x$. On the other hand, if $y \in \sO_x$,
then $\pi(y) = \pi(x)$, and hence $b = f(\pi(x)) = f(\pi(y)) = p(y)$.
Thus, $\sO_x \subseteq p^{-1}(b)$.

Strictly speaking, the translation function is defined on the bundle
$(X,\pi,X/G)$. Given any other $G$-bundle $(X,p,B)$, $p^{-1}(p(x))$
corresponds to the orbit $\sO_x$, for each $x \in X$. Through this
identification it is evidently possible to define a translation map
for the $G$-bundle $(X,p,B)$. However, to characterize a principal
bundle from its so-defined translation map, it is enough to consider
only the case $(X,\pi,X/G)$. Nevertheless, we want to explore which
properties on the action of $G$ on $X$ make the orbits (and hence the
fibers of a $G$-bundle) look alike, and vice versa. First of all,
saying that all the fibers of a given $G$-bundle are in bijective
correspondence, amounts all orbits of the action of $G$ on $X$ being
in bijective correspondence. Since the orbits are subsets of $X$, they
are naturally topological spaces with the relative topology.

So, what sort of actions carry mutually homeomorphic orbits?

By Corollary~\ref{cor:continuous-free-proper-action}, if the action is
free and proper, then $\{x\} \x G$ is homeomorphic to
$\{x\} \x \sO_x$ in $X \x X$, for every $x \in X$. As a result, every
orbit is homeomorphic to~$G$. This homeomorphism can be described as
follows: for every $x \in X$ one may define a surjective map
$\al_x : G \to \sO_x$ by $\al_x(g) = x\.g$; if the action is free,
every such map is injective. Reciprocally, if $\al_x$ is injective for
each $x \in X$, then the action $\al$ is free. For the sake of
completeness we state this as a lemma.

\begin{lemma} 
\label{lm:continuous-free-and-proper-imply-homeomorphic-fibers}
If $\al$ is a free action, the group $G$ is in bijective
correspondence with each orbit in $X$. Furthermore, for a continuous
principal (free and proper) action, every orbit is \emph{homeomorphic}
to~$G$.
\end{lemma}

We use the following notation. If $y \in X$ lies in the same orbit
as~$x$, $y = x\.g$ for some $g \in G$, then
$\al_x^{-1}(y) = g = \tau(x,y) = \tau_x(y)$, where
$\tau_x \: \sO_x \to G$ is the map obtained by fixing the first
coordinate in the argument of~$\tau$.

The reciprocal result to
Lemma~\ref{lm:continuous-free-and-proper-imply-homeomorphic-fibers}
will state that if every orbit of a continuous action is homeomorphic
to $G$ via $\al_x \: G \to \sO_x : \al_x(g) = x\.g$, then the action
is free and proper. We already know that the bijectivity of these
homeomorphisms $\al_x$ is enough to conclude the freeness of the
action. So far what we are missing is the properness of the action. By
Lemma~\ref{lm:proper-action-characterization}, the question to be
answered is, whether every
$\al_x^{-1} = \tau_x \: \sO_x \to G : \tau_x (x\.g) = g$ being a
homeomorphism is enough to conclude first, that $X \x_G X$ is closed
in $X \x X$ (or that $X/G$ is Hausdorff thanks to
Lemma~\ref{lm:DuKo}), and second, that this translation map
$\tau \: X \x_G X \to G$ is continuous.

A $G$-space for which the continuous action of $G$ on $X$ is free and
the translation function $\tau \: X \x_G X \to G$ is continuous
(equivalently, the canonical map is a homeomorphism), will be called a
\textit{quasi-principal $G$-space} (\textit{principal $G$-space} in
\cite{Husemoller}).%
\footnote{%
We decided to add the prefix `quasi' since from
Lemma~\ref{lm:proper-action-characterization} we know that a
quasi-principal $G$-space is just one property short of having a
proper action, i.e., the quotient topology on $X/G$ need not be
Hausdorff (or $X \x_G X$ closed in $X \x X$). By
Lemma~\ref{lm:proper-action-characterization}, this use of the prefix
`quasi' is equivalent to that given in
\cite[p.~13]{BaumHajacMatthesSzymanski}, in the case of a Hausdorff
base space.}
We call a continuous, free action whose translation function
$\tau \: X \x_G X \to G$ is continuous a \textit{quasi-principal}
action.

We reserve the term \textit{principal $G$-space} (in contrast
to~\cite{Husemoller}) for a $G$-space for which the continuous action
of $G$ on $X$ is free and proper. In particular, the quotient topology
on $X/G$ is Hausdorff. Furthermore, a $G$-bundle $(X,p,B)$ for which
$X$ is a quasi-principal $G$-space shall be called in these notes a
\textit{quasi-principal $G$-bundle} (\textit{principal $G$-bundle}
in~\cite{Husemoller}).

Observe that in a quasi-principal $G$-bundle, the base space $B$ is
not required to be Hausdorff. Needless to say, if we ask all of our
topological spaces to be Hausdorff (as is the case when working with
smooth manifolds, for example), we shall not have to worry about the
quasi situation.

A \textit{fiber bundle} (with \textit{fiber} a topological space $F$)
is a bundle $p \: E \to B$, such that $p$ defines a \textit{local
product structure} over $E$ in the sense that each point of~$B$ has a
neighborhood $U$ and a homeomorphism $\psi \: U \x F \to p^{-1}(U)$,
in such a way%
\footnote{%
In some references (see \cite{BaumHajacMatthesSzymanski} and
references therein), the concept of \textit{locally trivial bundle}
requires that the fiber depend on the base point, allowing, as in the
case of a base space with several connected components, divers fibers
on different base points.}
that $p \circ \psi \: U \x F \to U$ is the projection $\pr_1$ on the
first factor. Note that $p$ is an open map; indeed, we can shrink the
neighborhoods $U$ to open sets $O$ and restrict the homeomorphism
$\psi$ to $\psi|_O : O \x F \to p^{-1}(O)$, and with
$p = \pr_1 \circ (\psi|_O)^{-1}$ we conclude that $p$ is a locally
open map. If $A$ is an open set in~$E$, each set $p^{-1}(O_{p(x)})
\cap A$ is open. From the relation
$p(A) = p\bigl( \bigcup_{x\in A} p^{-1}(O_{p(x)}) \cap A \bigr) 
= \bigcup_{x\in A} p\bigl( p^{-1}(O_{p(x)}) \cap A \bigr)$ the claim
that $p$ is an open map follows. For each $b \in B$, the map
$\psi|_b \: \{b\} \x F \to p^{-1}(b)$ is a homeomorphism. By an abuse
of notation, we regard it as a map $\psi|_b \: F \to p^{-1}(b)$. In
this way, if a point $b \in B$ belongs to $U_i \cap U_j$, then there
is a homeomorphism
$\psi_i|_b \circ \psi_j|_b^{-1} : p^{-1}(b) \mapsto p^{-1}(b)$, which
in general is not the identity on every $p^{-1}(b)$.

Imagine now that we want to take advantage of this product structure
to define an action $\al \: E \x G \to E$ by taking
$F = G$. If $(e,g) \in E \x G$ then $p(e) \in B$. Let $(U_i,\psi_i)$
be a pair in the local decomposition of $p$ with $p(e) \in U_i$. It is
tempting to define $\al(e,g) := \psi_i(p(e),g)$; but there is the
question of whether this is well defined. Consider another pair
$(U_j,\psi_j)$ with $p(e) \in U_i \cap U_j$; then
$\psi_j(p(e),g) = \psi_i(p(e),g)$ if and only if
$\psi_j|_{p(e)}(g) = \psi_i|_{p(e)}(g)$, if and only if
$g = \psi_j|_{p(e)}^{-1} \circ \psi_i|_{p(e)}(g)$. Hence, it need not
be true in general that a product structure on a fiber bundle defines
an action of $G$ on $E$, at least in the canonical way intended in
this paragraph. So we ask for a little more structure.

A \textit{quasi-principal fiber $G$-bundle}%
\footnote{%
We have decided to keep the adjective `fiber' to emphasize its
dependence on the chosen fiber. We use the ornament `quasi' since by
Proposition~\ref{pr:quasi-to-principal} it is required to have a
Hausdorff base for the action to be principal.
Proposition~\ref{pr:quasi-to-principal} will show that in a
quasi-principal fiber $G$-bundle with Hausdorff base, the action of
$G$ on~$X$ is indeed principal.}
is a fiber bundle $(X,p,B,G)$ for which there is a continuous right 
action of $G$ on $X$, and a coordinate representation
$\{(U_i, \psi_i)\}$ such that
\begin{equation}
\label{eq:principal-condition} 
\psi_i(b,gh) = \psi_i (b,g)\.h, 
\word{for all}  b \in U_i, \ g,h \in G.
\end{equation}

How close is the action of a quasi-principal fiber $G$-bundle to be
principal (free and proper)?

First, assume that $x\.g = x$ for some $x$ in $X$ and $g$ in $G$. Let
$b = p(x)$ and let $U \subset B$ with $b \in U$. Then
$x \in p^{-1}(U) = \psi(U \x G)$, and so, $x = \psi(b,h)$ for some
$h$ in $G$. In this way, 
$\psi(b,h) = x = x\.g = \psi(b,h)\.g = \psi(b,hg)$. Since $\psi$ is a
homeomorphism, it follows that $h = hg$, and hence $g = e$. Therefore,
the action is free.

Second, in a (quasi) principal fiber $G$-bundle the orbits coincide
with the fibers. Indeed, if $y = x\.g$ for some $g \in G$, then as in
the previous paragraph, $x = \psi(p(x),h)$ for some $h \in G$ and some
pair $(U,\psi)$. Here
$p(y) = p(x\.g) = p(\psi(p(x),h)\.g) = p(\psi(p(x),hg)) = p(x)$.
Meaning that every point in the orbit of~$x$ belongs to the same fiber
as~$x$. On the other hand, if $y$ is in the same fiber as~$x$, i.e.,
$p(y) = p(x)$, then taking a pair $(U,\psi)$, there are $g,h \in G$
with $y = \psi(p(x),h)$ and $x = \psi(p(x),g)$. Thus,
$y\.(h^{-1}g) = \psi(p(x),h)\.(h^{-1}g) = \psi(p(x),g) = x$.

So let us define $f \: X/G \to B$ by sending each orbit to its base
point in~$B$. In this way $f$ is an invertible map, and furthermore,
$p(x) = b$ entails $f(\pi(x)) = b$.

In $X/G$ there is the quotient topology, whereby a set of classes is
open in $X/G$ if and only if the union of all those classes is open in
$X$. Let $U \subset B$ be an open set. The set $f^{-1}(U)$ is the
union of the orbits (fibers) over all points in~$U$, which is open
in $X$, since this set is precisely
$p^{-1}(U) = \pi^{-1}(f^{-1}(U))$. Therefore $f^{-1}(U)$ is open
in~$X/G$ and thus $f$ is continuous.

Now let $O \subset X/G$ be an open set, so that $\pi^{-1}(O)$, which
is the union of all the classes in~$O$, is open in $X$. Since
$O = \pi(\pi^{-1}(O))$ one obtains
$f(O) = f(\pi(\pi^{-1}(O))) = p(\pi^{-1}(O))$. Since $p$ is an open
map, as already noted, so also is $f$, and therefore $f$~is a
homeomorphism. The next result follows immediately.

\begin{proposition} 
\label{pr:existence-of-f}
Any quasi-principal fiber $G$-bundle $(X,p,B,G)$ is a $G$-bundle.
\qed
\end{proposition}

Back to our last question, how close is the action of a (quasi)
principal fiber $G$-bundle to be free and proper? By
Lemma~\ref{lm:proper-action-characterization}, properness accounts for
$X/G$ to be Hausdorff and the translation map to be continuous. Since
$X/G$ is homeomorphic to $B$, it is enough to request that our base
space $B$ be Hausdorff in order to have a proper action. We conclude
the following.

\begin{proposition} 
\label{pr:quasi-to-principal}
In a quasi-principal fiber $G$-bundle with Hausdorff base, the
action of $G$ on $X$ is principal (free and proper).
\qed
\end{proposition}

If $(X,p,B)$ is a principal $G$-bundle and the topological group $G$
is Hausdorff, taking advantage of the local homeomorphisms
$\psi \: U \x G \to p^{-1}(U)$, we conclude that $X$ must be 
Hausdorff, too.

\begin{proposition} 
\label{pr:B-and-G-Hausdorff-implies-X-Hausdorff}
If $(X,p,B)$ is a principal $G$-bundle and the topological group $G$
is Hausdorff, then the total space $X$ is Hausdorff.
\end{proposition}

\begin{proof}
Let $x \neq y$ be two points in $X$. If $p(x) \neq p(y)$, then we can
consider disjoint open sets $U_x \ni p(x)$, $U_y\ni p(y)$ in $B$
and homeomorphisms $\psi_x \: U_x \x G \to p^{-1}(U_x)$,
$\psi_y \: U_y \x G \to p^{-1}(U_y)$ to conclude that the disjoint
open sets $p^{-1}(U_x)$ and $p^{-1}(U_y)$ separate the points $x$
and~$y$.

In the case $p(x) = p(y)$, $x$ and $y$ belong to the same orbit~$\sO$.
By
Lemma~\ref{lm:continuous-free-and-proper-imply-homeomorphic-fibers},
there is a homeomorphism $\vf \: \sO \to G$, while $G$ is Hausdorff.
Considering disjoint open sets $V_x$ and $V_y$ in $G$ that separate
the image points $\vf(x)$ and $\vf(y)$ in~$G$. By shrinking the
homeomorphism $\psi \: U \x G \to p^{-1}(U)$, with $p(x) = p(y) \in
U$, to the respective domains $U \x V_x$ and $U \x V_y$, we obtain
disjoint open sets $\psi(U \x V_x)$ and $\psi(U \x V_y)$ that separate
$x$ and $y$ in~$X$.
\end{proof}

\begin{remark}
As was noted before, in any fiber bundle $(X,p,B,G)$,
$\psi|_b \: G \to p^{-1}(b)$ is a homeomorphism. Using the translation
map of Section~\ref{sec:trans-map-topol-spaces} we can observe the
following. For a given $x \in p^{-1}(b)$, consider the bijective map
$u \: G \to p^{-1}(b) : g \mapsto x\.g$. Its inverse function $u^{-1}$
is $x' \mapsto \tau(x,x')$, which is continuous thanks to
Lemma~\ref{lm:proper-action-characterization}. Thus $u$ is a
homeomorphism and the translation map (restricted on its first entry
to a point~$x$) provides the homeomorphism between the fiber over
$p(x)$ and the group~$G$.
\end{remark}

Next we want to reproduce
\cite[Proposition~1.20]{BaumHajacMatthesSzymanski} about local
triviality being equivalent to the existence of local sections. By the
paragraph before Proposition~\ref{pr:existence-of-f}, it is enough to
consider the case $(X,\pi,X/G)$. If the action of $G$ on $X$ is
continuous and free and if $(X,\pi,X/G)$ has a coordinate
representation satisfying \eqref{eq:principal-condition}, then for
every $x \in X$ there is an open set $U \ni \pi(x)$ and a
homeomorphism $\psi \: U \x G \to \pi^{-1}(U)$ such that
$\pr_1 = \pi \circ \psi$. Define the continuous map
$s_\psi \: U \to \pi^{-1}(U) : u \mapsto \psi(u,e)$. Note that
$\pi \circ s_\psi = 1_U$ and $x \sim s_\psi(\pi(x))$, so
$s_\psi \cdot \tau(x,\psi(u,e))$ is a (continuous) local section for
$(X,\pi,X/G)$ at~$\pi(x)$.

\begin{proposition} 
\label{pr:local-sections}
If the action of $G$ on $X$ is continuous and free, then $(X,\pi,X/G)$
is a quasi-principal fiber $G$-bundle if and only if there is a
continuous local section through each point of~$X$.
\end{proposition}

\begin{proof}
It remains to show the reciprocal statement, so let $s \: U \to X$ be
a continuous local section through $x \in X$ and define the
continuous map $\psi_s \: U \x G \to \pi^{-1}(U)$ by
$\psi_s(u,g) := s(u)\.g$ with
$\pi(s(u)) = \pi(s(u)\.g) = \pi(\psi_s(u,g)) = u$. Evidently, $\psi_s$
satisfies~\eqref{eq:principal-condition}. If
$\psi_s(u,g) = \psi_s(u',g')$, then applying $\pi$ gives $u = u'$, and
because the action is free, $\psi_s$ is an injective map. If
$y \in \pi^{-1}(U)$, then $\pi(y) = \pi \circ s(\pi(y))
= \pi(\psi_s(\pi(y),e))$, meaning that there is a $g \in G$ with
$y = \psi_s(\pi(y),e) \cdot g = \psi_s(\pi(y),g)$. Hence, $\psi_s$ is
surjective as well. The continuity of
$\psi_s^{-1} \: \pi^{-1}(U) \to U \x G$ follows from the relation
$y = \psi_s(\pi(y),e) \cdot \tau(\psi_s(\pi(y),e),y)$, used to show
the surjectivity of~$\psi_s$.
\end{proof}

From this section we draw the following conclusion: the existence of a
continuous translation map does not characterize a topological
principal fiber $G$-bundle, in the more general sense it is missing
that the base space be Hausdorff.


\subsection{The Milnor construction} 
\label{ssc:Milnor}

The objective of this section is to present an example of a bundle
$(E_G,p,B_G)$ that is a quasi-principal or principal $G$-bundle,
depending on the properties of the topological group~$G$. This example
is of particular interest since in \cite{Milnor} (see also
\cite{Husemoller}) it is shown that the said bundle is
\textit{universal}, in the sense that for any (quasi) principal
$G$-bundle $(X,\pi,B)$ there is a continuous map $f \: B \to B_G$
(unique up to homotopy) such that the pull-back bundle
$(f^*(E_G),\pi',B)$ is homeomorphic to $(X,\pi,B)$; see
\cite[Proposition 10.6(1), p.~52 and Theorem~12.2,
p.~55]{Husemoller}. The base space $B_G$ is called the
\textit{classifying space} of the group G. The name is motivated by
the circumstance that the map $f \mapsto (f^*(E_G),\pi',B)$ induces a
bijection between homotopy classes of continuous maps from $B$
to~$B_G$ and homeomorphism classes of (quasi) principal $G$-bundles
with base space $B$: see \cite[Proposition 10.6(2), p.~52 and
Theorem~12.4, p.~56]{Husemoller}.

Let $G$ be a topological group and consider the set $E_G$ of formal
sequences $(t_ng_n)_n = (t_0g_0,t_1g_1,\dots)$ with $g_n \in G$ and
$t_n \in[0,1]$, with only finitely many nonzero~$t_n$ and
$\sum t_n = 1$.%
\footnote{%
The condition $\sum t_n = 1$ is requested in order to have a numerable
$G$-bundle: see~\cite[Definition~9.1, p.~48]{Husemoller}.}
Furthermore, if any $t_m = 0$, it is assumed that $0 g_m = 0 h_m$, for
any $g_m,h_m \in G$. The topology on $E_G$ is the coarsest one making
continuous the functions
$$
t_m \: E_G \to [0,1] : (t_ng_n)_n \mapsto t_m, \word{and}
g_m \: t_m^{-1}((0,1]) \to G : (t_ng_n)_n \mapsto g_m,
$$
with a tiny abuse of notation. An open set in $E_G$ is an arbitrary
union of finite intersections of sets of the form $t_m^{-1}(I)$ and
$g_n^{-1}(O)$ with $I$ open in $[0,1]$ and $O$ open in~$G$. The action
$\al \: E_G \x G \to E_G$ is given by
$(t_ng_n)_n \cdot g = (t_ng_ng)_n$.

It follows that 
$t_m \bigl((t_ng_n)_n\.g\bigr) = t_m \bigl((t_ng_ng)_n\bigr) = t_m$
and $g_m \bigl((t_ng_n)_n\.g\bigr) = g_m \bigl((t_ng_ng)_n\bigr) 
= g_m g = g_m \bigl((t_ng_n)_n\bigr)g$. These relations can be written
as $t_m \circ \al = t_m \circ \pr_1 \: E_G \x G \to [0,1]$ and
$g_m \circ \al = \mu \circ(g_m \x 1_G) \: E_G \x G \to G$, with
$\mu$ the multiplication in~$G$. As a result, $\al$ is a continuous
action and $E_G$ is a $G$-space. It is easily seen that this action
is free.

Denote by $B_G$ the quotient space $E_G/G$ under this action. The
translation function $\tau \: E_G \x_G E_G \to G$ is continuous.
Indeed, if $(s_nh_n)_n = (t_ng_n)_n\.g = (t_ng_ng)_n$ in $E_G$, then
there is some~$m$ with $s_m = t_m > 0$ and
$\tau\bigl( (t_ng_n)_n, (s_nh_n)_n \bigr) = g 
= g_m \bigl( (t_ng_n)_n \bigr)^{-1} g_m \bigl( (s_nh_n)_n \bigr)$,
meaning that $\tau$ restricted to the open set
$(t_m^{-1}(0,1] \x t_m^{-1}(0,1]) \cap(E_G \x_G E_G)$ is continuous.
Since the open sets $t_m^{-1}(0,1] \x t_m^{-1}(0,1]$ cover
$E_G \x_G E_G$, it follows that $\tau$ is a continuous function, and
hence the action is quasi-principal.

Before looking into a Hausdorff structure on $B_G$ and whether
$(E_G,p,B_G)$, with $p$ the canonical projection, is a principal
$G$-bundle, let us ponder a Hausdorff structure on~$E_G$. Two points
$(t_ng_n)_n$ and $(s_nh_n)_n$ are different in $E_G$ if
\textit{either} there is an $m$ with $t_m \neq s_m$; \textit{or}
$t_n = s_n$ for all~$n$ but there is some~$m$ with $t_m > 0$ and
$g_m \neq h_m$ in $G$. If the first case, one can take open and
disjoint intervals $I_1$, $I_2$ in $[0,1]$ with $t_m \in I_1$ and
$h_m \in I_2$. Since $(t_n g_n)_n \in t_m^{-1}(I_1)$,
$(s_nh_n)_n \in t_m^{-1}(I_2)$ and
$t_m^{-1}(I_1) \cap t_m^{-1}(I_2) = \emptyset$, one can then separate
those two points in $E_G$ with open sets. In the second case, in order
to mimic the previous argument, we shall demand that $G$ be Hausdorff.

\begin{proposition} 
\label{pr:Milnor-EG-iff-G}
In the Milnor construction, $G$ is Hausdorff if and only if $E_G$ is
Hausdorff.
\end{proposition}

\begin{proof}
Assume that $G$ is Hausdorff and take two points 
$(t_ng_n)_n \neq (s_nh_n)_n$ in~$E_G$. We are assuming that the second
condition holds, so $t_n = s_n$ for all~$n$ but there is an~$m$ with
$t_m > 0$ and $g_m \neq h_m$ in $G$. There are disjoint open sets
$O_1$, $O_2$ in $G$ with $g_m \in O_1$, $h_m \in O_2$. Since
$(t_ng_n)_n \in g_m^{-1}(O_1)$, $(s_nh_n)_n \in g_m^{-1}(O_2)$ and
$g_m^{-1}(O_1) \cap g_m^{-1}(O_2) = \emptyset$, it follows that once
again we can separate the two points $(t_ng_n)_n$ and $(s_nh_n)_n$ in
$E_G$ with open sets. So far we have shown that if $G$ is Hausdorff,
then $E_G$ is Hausdorff, too.

Now assume that $E_G$ is Hausdorff and suppose that $g \neq h$ in $G$.
Consider the two different elements $\unl{g} = (g,0,0,\dots)$,
$\unl{h} = (h,0,0,\dots)$ in~$E_G$. Since
$t_n(\unl{g}) = t_n(\unl{h})$ for all~$n$, the elements $\unl{g}$,
$\unl{h}$ cannot be separated in $E_G$ by open sets of the form
$t_n^{-1}(I)$ with $I$ open in $[0,1]$. So the only possibility to
separate $\unl{g}$ and $\unl{h}$ in $E_G$ would be with disjoint open
sets of the form $g_0^{-1}(O_1) \ni \unl{g}$ and
$g_0^{-1}(O_2) \ni \unl{h}$. Were there
$k \in O_1 \cap O_2 \subseteq G$, then
$g_0^{-1}(k) \subseteq g_0^{-1}(O_1) \cap g_0^{1}(O_2) = \emptyset$,
which is absurd. In this way, there exist disjoint open sets
$O_1 \ni g$ and $O_2 \ni h$ in $G$ that separate $g$ and~$h$; thus,
$G$ is Hausdorff.
\end{proof}

Let us now deal with the Hausdorff topology on $B_G$; first, we state
some facts and notation. Let $x = (t_ng_n)_n$ be an element of $E_G$.
Since only finitely many $t_n$ are nonzero, there is a greatest index
$m_x$ and a least index $M_x$ such that $t_j = 0$ for all $j < m_x$ or
$j > M_x$. This allows us to work with the \textit{truncation}
$[x] = (t_{m_x} g_{m_x},\dots, t_{M_x} g_{M_x})$, where we allow that
possibly some of the $t_i$ might be zero for $m_x < i < M_x$. Also
note that if $x \sim y$ because $t_n(x) = t_n(y)$ for every $n$, we
conclude that $m_x$ and $M_x$ are the same for any $y$ in the same
orbit $\sO$ as~$x$. Furthermore,
$[y] = (t_{m_\sO} g_{m_\sO},\dots,t_{M_\sO} g_{M_\sO}) \. \tau(x,y)$,
for each $y \in \sO$, i.e, $g_i(y) = g_i(x) \cdot \tau(x,y)$. For any
orbit $\sO$ we may choose a representative $x_e$ of the form
$[x_e] = (t_m e, t_{m+1} g_{m+1}, \dots, t_M g_M)$, that is,
$g_m(x_e) = e$, simply by replacing any $x \in \sO$ by
$x\.g_m(x)^{-1}$.

Let $U$ be an open set in $B_G$ and consider the map
$\psi \: U\x G \to p^{-1}(U)$ given by 
$(p(x),g) \mapsto x_e\.g = x\.g_m(x)^{-1} g = x\.\tau(x,x_e)g$.
Trivially, $\psi(p(x),gh) = \psi(p(x),g)h$ for any $g,h \in G$. If
$x \in p^{-1}(U)$, then $x = x_e\.g_m(x) = \psi(p(x_e),g_m(x))
= \psi(p(x),g_m(x))$, making $\psi$ a surjective map such that
$p \circ \psi = \pr_1$. Also, if 
$x_e\.g = \psi(p(x),g) = \psi(p(y),h) = y_e\.h$, for $x$ and $y$ in
the same orbit, then $x_e = y_e$ and $g = h$, showing that $\psi$ is
an injective map. Since $\psi^{-1} \: p^{-1}(U) \to U \x G$ is given
by $x \mapsto (p(x),g_m(x))$, it is a continuous map.

To verify that $\psi$ is continuous, let $V = V' \cap p^{-1}(U)$ be an
open set in $p^{-1}(U)$, with $V'$ an open set in $E_G$. To verify
that $\psi^{-1}(V)$ is open in $U \x G$, it is enough to do it for
sets $V'$ of the form $t_i^{-1}(I)$ and $g_j^{-1}(O)$ with $I$ open in
$[0,1]$ and $O$ open in~$G$. If $V' = t_i^{-1}(I)$, then
$(p(x_e),g) \in \psi^{-1}(t_i^{-1}(I) \cap p^{-1}(U))$ if and only if
$x_e\.g_m(x_e)^{-1} g \in t_i^{-1}(I) \cap p^{-1}(U)$, which in turn
is equivalent to $x_e \in t_i^{-1}(I)\cap p^{-1}(U)$. So,
$\psi^{-1}(t_i^{-1}(I) \cap p^{-1}(U)) = (p(t_i^{-1}(I))\cap U) \x G$,
open in $U \x G$. If $V' = g_j^{-1}(O)$, then
$(u,g) = (p(x_e),g) \in \psi^{-1}(g_j^{-1}(O)\cap p^{-1}(U))$ for some
$x \in E_G$, if and only if $x_e\.g \in g_j^{-1}(O) \cap p^{-1}(U)$.
Since this is equivalent to $g_j(x_e) \in Og^{-1}$ open in $G$, the
representative $x_e$ for~$u$ must belong to the set
$g_j^{-1}(Og^{-1})$, open in $E_G$, and $u \in p(g_j^{-1}(Og^{-1}))$,
open in~$U$. We conclude that $\psi$ is a continuous map and that
$(E_G,p,B_G)$ has a local product structure.

By Proposition~\ref{pr:local-sections}, there is a local continuous
section through each point $x \in E_G$. Indeed, let
$x = x_e\.g_m(x)$ be fixed in~$E_G$, and consider
$s \: U \to p^{-1}(U) : p(y)\mapsto \psi(p(y),e)$, with
$s(p(x)) = x_e$, meaning that $s\.g_m(x)$ is a local section with 
$(s\.g_m(x))(p(x)) = x_e\.g_m(x) =  x$.

\begin{proposition} 
\label{pr:Milnor-BG-iff-G}
In the Milnor construction, if $G$ is Hausdorff, then $B_G$ is
Hausdorff and $(E_G,p,B_G)$ is a principal $G$-bundle.
\end{proposition}

\begin{proof}
Following Lemma~\ref{lm:proper-action-characterization}, if $B_G$ is
Hausdorff then $(E_G,p,B_G)$ is a principal $G$-bundle. By
Lemma~\ref{lm:DuKo}, $B_G$ is Hausdorff if and only if $E_G \x_G E_G$
is closed in $E_G \x E_G$, so let
$(x,y) = \bigr( (t_ng_n)_n, (s_nh_n)_n \bigl)$ be an element of
$E_G \x E_G$ not in $E_G \x_G E_G$.

Since only finitely many $t_n$ and $s_n$ are nonzero, there is an
index~$m$ such that $t_r = s_r = 0$ for all $r > m$; this allows us to
work with their \textit{$m$-th truncations}
$x_{[0,m]} = (t_0g_0,\dots,t_mg_m)$ and
$y_{[0,m]} = (s_0h_0,\dots, s_mh_m)$. Observe that
$t_n(x\.g) = t_n(x)$ for any $g \in G$.

There are thus two possibilities for $y_{[0,m]} \neq x_{[0,m]}\.g$ for
all $g \in G$. First, there is an index $n \in \{0,\dots,m\}$ such
that $t_n \neq s_n$; or second, $t_n = s_n$ holds for all
$n \in \{0,\dots,m\}$ and for each $g \in G$, there is some
$n_g \leq m$ with $t_{n_g} = s_{n_g} > 0$ and $h_{n_g} \neq g_{n_g}g$.

If the first possibility holds, there are disjoint open sets 
$I_1 \ni t_n$ and $I_2 \ni s_n$ in $[0,1]$. If $(z,w)$ belongs to the
open set $t_n^{-1}(I_1) \x t_n^{-1}(I_2) \subseteq E_G \x E_G$ and
there is $h \in G$ with $w = z\.h$, then
$t_n(w) = t_n(z) \in I_1 \cap I_2$, which is impossible. Hence
$(t_n^{-1}(I_1) \x t_n^{-1}(I_2)) \cap (E_G \x_G E_G) = \emptyset$.

Under the second possibility, $x_{[0,m]} = (t_0g_0,\dots,t_mg_m)$ and
$y_{[0,m]} = (t_0h_0,\dots,t_mh_m)$, where for each $g \in G$ there is
an index $n_g \leq m$ with $t_{n_g} > 0$ and $h_{n_g} \neq g_{n_g}g$.
We may assume, without loss of generality, that $t_j(x) = t_j(y) > 0$
for all $j = 0,\dots,m$ by reordering them if need be and taking a
smaller~$m$.

Since $E_G$ is Hausdorff, there are disjoint open sets $V_x \ni x$ and
$V_y \ni y$. By shrinking the respective open sets $U_x = p(V_x)$ and
$U_y = p(V_y)$ in $B_G$ if necessary, we know there are local product
homeomorphisms $\psi_x \: U_x \x G \to p^{-1}(U_x)$ and
$\psi_y \: U_y \x G \to p^{-1}(U_y)$. If $U_x$ and $U_y$ can be shrunk
even further so as to have $U_x \cap U_y = \emptyset$, then
$(x,y) \in p^{-1}(U_x) \x p^{-1}(U_y)$ and 
$(p^{-1}(U_x) \x p^{-1}(U_y)) \cap (E_G \x_G E_G) = \emptyset$.

So we are left with the case where however much we shrink
$U_x \ni p(x)$ and $U_y \ni p(y)$, $U_x \cap U_y \neq \emptyset$
persists. We shall prove that such an assumption leads to a
contradiction. If such a pair of open sets exists, for each of them
there are continuous sections $s_x \: U_x \to p^{-1}(U_x)$ and
$s_y \: U_y \to p^{-1}(U_y)$ satisfying $s_x(p(x)) = x$ and
$s_y(p(y)) = y$.  For any point $z \in U_x \cap U_y$, 
$s_x(z) \sim s_y(z)$ in $p^{-1}(U_x \cap U_y) \subseteq E_G$.
Neither $p(x)$ nor $p(y)$ belongs to $U_x \cap U_y$, by assumption.
Define a continuous function $h \: U_x \cap U_y \to G$ by
$h(z) := \tau(s_x(z), s_y(z))$. On $U_x \cap U_y$ both sections
$s_x|_U \cdot h \: U \to p^{-1}(U)$ and $s_y|_U \: U \to p^{-1}(U)$
coincide. By taking successively smaller sets, we can build two nets
of open sets $\{O_i\}_{i\in I}$ and $\{O'_i\}_{i\in I}$ directed by 
inverse inclusion and satisfying $p(x) \in O_i \subseteq U_x$,
$p(y) \in O'_i \subseteq U_y$, with $O_i \cap O'_i \neq \emptyset$. As
a result, there is a net $\{z_i\}_{i\in I} \subseteq B_G$ with each
$z_i \in O_i \cap O'_i$, such that $\{z_i\}_{i\in I}$ converges in
$B_G$ to both $p(x)$ and~$p(y)$. By continuity,
$\{s_x(z_i)\}_{i\in I}$ converges to $x$ in $E_G$ and
$\{s_y(z_i)) = s_x(z_i) \cdot h(z_i))\}_{i\in I}$ converges to~$y$,
contradicting that $x$ and $y$ belong to different orbits. We conclude
that $B_G$ is Hausdorff.
\end{proof}

We wish to move from the topological to the smooth situation. As an
intermediate step, we shall study the case of a discrete group acting
on a smooth manifold, leaving the case of a general Lie group for
Section~\ref{sec:Husemoller-principal-bundles}.


\section{A discrete group acting on a smooth manifold} 
\label{sec:discrete-case}

In this section we deal with the case of a discrete group acting on a
smooth manifold.%
\footnote{Our manifolds are Hausdorff and second countable, to ensure
for example, the existence of partitions of unity. In
\cite[Problem~1-5, p.~30]{Lee} it is requested to show that for a
locally Euclidean Hausdorff space, second countable is equivalent to
paracompact and having countably many components.}
From the previous section, we know that a continuous translation map
with Hausdorff quotient space characterizes a principal fiber
$G$-bundle in the topological case. When we replace the topological
space $X$ by an smooth manifold, we wish to obtain more: a smooth
structure on the quotient space and some sort of smoothness for the
translation map. There are a few types of actions available in the
literature for a discrete group acting on a space.

Since we now aim to characterize principal fiber $G$-bundles whose
total space is a smooth manifold, we ought to demand that our action
$X \x G \to X$ be smooth. Thus, for each $g \in G$, the map
$X \to X : x \mapsto x\.g$ will be a smooth map, and actually a
diffeomorphism with inverse $x \mapsto x\.g^{-1}$. However, for now we
still use continuous actions and will not require a smooth action
until Theorem~\ref{thm:covering-principal-is-principal-discrete-case}.

We work with a free action to ensure that the translation map
$\tau \: X \x_G X \to G$ is well defined. The action of a discrete
group $G$ on smooth manifold $X$ is called a
\textit{covering-space action}%
\footnote{As indicated by J. M. Lee in 
\url{https://math.stackexchange.com/questions/1082834},
this named was suggested by A. Hatcher in \cite{Hatcher}, in
opposition to the term ``properly discontinuous'' used for example
in~\cite{DoCarmo}.}
if every $x \in X$ has a neighborhood $U \subset X$ such that 
$U \cap U\.g = \emptyset$ for all $g \neq e$ in~$G$. Note that
covering-space actions are free.

\begin{example} 
\label{ex:covering-not-proper}
Consider the action of $G = \bZ$ on $X = \bR^2 \less \{(0,0)\}$,
given by $\al: X \x G \to X \x X : (x,y)\.n = (2^n x, 2^{-n} y)$. By
staying away from the nonzero numbers $\{x/2, y/2, 2x, 2y\}$, we can
construct an open box $U$ centered at $(x,y)$ such that
$U \cap U\.n = \emptyset$ for every $n \in \bZ \less \{0\}$, making it
a covering-space action.

Now consider the compact set $K = \set{(x,y) : \max(|x|,|y|) = 1}$.
The subset $K \x K \subseteq X \x X$ is compact. The sequence
$((2^{-n},1),n)_{n\in\bN} \subseteq \al^{-1}(K \x K)$ is not compact
in $X \x G$ since it has no convergent subsequence. By
Lemma~\ref{lm:proper-map}, the action $\al$ is not proper.

A more elaborate example is the Kronecker foliation of the $2$-torus
$X = \bT^2 \cong \bR^2/\bZ^2$, $G = \bR$,
$X \x G \to X : ((x,y),t) \mapsto (x + t, y + \theta t)$ with $\theta$
irrational. In \cite{Ellwood} this action is shown to be free but not
proper and the quotient space is non-Hausdorff, having as the unique
open sets $\emptyset$ and $X/G$ (see
Lemma~\ref{lm:proper-action-characterization}).
\end{example}

In \cite[p.~143]{NovikovTaimanov} -- see also \cite[Lemma~21.11]{Lee}
-- an action is called \textit{discrete} if as well as being a
covering-space action it satisfies the following: for any two points
$x,y \in X$ with different orbits there are neighborhoods $U$ of~$x$
and $V$ of~$y$ such that the orbits of those neighborhoods are
disjoint.

We require more manageable ways to deal with proper actions. We now
offer equivalent definitions taken from different references. Even
though we work with discrete groups in this section, in
Section~\ref{sec:Husemoller-principal-bundles} we shall move to a less
simple structure on~$G$. For example, it will follows that every
continuous action of a compact Lie group on a smooth manifold is
proper.

We recall (from the proof of
Lemma~\ref{lm:proper-action-characterization}) that the action is
proper if and only if the extension of its canonical map,
$\wt{\ga} \:  G \x X \to X \x X : (g,x) \mapsto (g\.x,x)$ is a proper
map.

\begin{lemma} 
\label{lm:equivalent-proper-actions}
Let $G$ be a Lie group acting continuously%
\footnote{We do not require the action to be smooth. Also what we ask
of $G$ and~$X$ is that they are Hausdorff and second countable. Using
partitions of unity, $X$ can be endowed with a Riemannian metric, so
that this metric topology is equivalent to the original topology on
the manifold.}
on a smooth manifold $X$. The following statements are equivalent.
\begin{enumerate}[noitemsep,label={\textup{(\roman*)}}]
\item 
The action is proper.
\item 
For every compact set $L \subseteq X \x X$, $\wt{\ga}^{-1}(L)$ is
compact in $G \x X$.
\item 
For every compact set $K \subseteq X$, the set
$G_K = \set{g \in G : (K\.g) \cap K \neq\emptyset}$ is compact.
\item 
If $(x_n)$ is a convergent sequence in $X$ and $(g_n)$ is a sequence
in $G$ such that $(x_n\.g_n)$ is also convergent in $X$, then there is
a convergent subsequence of~$(g_n)$.
\end{enumerate}
\end{lemma}

\begin{proof} 
By Lemma~\ref{lm:proper-map}, (i) and (ii) are equivalent.

Assume $\wt{\ga}$ is a proper map and let $x = \lim x_n$,
$x' = \lim x_n\.g_n$ be limit points in $X$ as in~(iv). Let $U$ and
$U'$ be respective neighborhoods of $x$ and $x'$, both with compact
closure. For $n$ large enough, every point $(x_n, x_n\.g_n)$ belongs
to the compact set $\ovl{U} \x \ovl{U'}\subset X \x X$. By~(ii), its
inverse image under $\wt{\ga}$ is compact in $X \x G$, so there exists
a convergent subsequence, $(x_{n_k}, g_{n_k})$ in $X \x G$. In
particular, $(g_{n_k})$ is convergent in~$G$. This shows that (ii)
implies~(iv).

Let $K \subseteq X$ be a compact set and consider a sequence
$(g_n) \subseteq G_K$. Since for each~$n$,
$(K\.g_n) \cap K \neq \emptyset$, there are sequences $(x_n)$ and
$(x'_n)$ in $K$ with $x_n\.g_n = x'_n$. By passing to appropriate
subsequences, we can assume both $(x_n)$ and $(x'_n)$ converge in~$K$.
By~(iv) there is a convergent subsequence for $(g_n)$, meaning that
$G_K$ is compact. Thus, (iv) implies~(iii).

Let $L \subseteq X \x X$ be a compact set; both projections
$\pr_1(L)$ and $\pr_2(L)$ are compact in~$X$, as is
$L' = \pr_1(L) \cup \pr_2(L)$. By~(iii), $G_{L'}$ is compact in $G$.
Now, $\wt{\ga}^{-1}(L)$ is a closed set contained in
$\wt{\ga}^{-1}(L' \x L') 
= \set{(x,g) \in X \x G : (x,x\.g) \in L' \x L'}$. If both $x$ and
$x\.g$ belong to $L'$, then $x\.g \in L'\.g \cap L' \neq \emptyset$,
and so $g \in G_{L'}$. Thus $\wt{\ga}^{-1}(L) \subseteq L' \x G_{L'}$
is also compact. Hence, (iii) implies~(ii).
\end{proof}

A nice application of this lemma is the following result. For the
simplest of the (smooth) principal bundles, it states that the action
is proper and free, and could be seen as a motivation for the work in
Section~\ref{sec:Husemoller-principal-bundles}.

\begin{proposition} 
\label{pr:action-of-G-on-H-is-proper-free}
Let $G$ be a Lie group and $H \subset G$ be a closed subgroup. The
action $G \x H \to G$ given by multiplication is proper and free.
\end{proposition}

\begin{proof}
It follows straight from the definition that the action is free
($gh = g$ implies $h = e$).

For the proper part we use Lemma~\ref{lm:equivalent-proper-actions}.
Let $K \subseteq G$ be a compact subset and consider
$H_K = \set{h \in H : (K\.h) \cap K \neq \emptyset}$; we must show
that $H_K$ is compact. Let $(h_n)$ be a sequence in $H_K$. There exist
two sequences $(l_n)$ and $(k_n)$ in $K$ with $l_n h_n = k_n$. Since
$K$ is compact, by passing to appropriate subsequences we can assume
there are $l,k \in K$ with $l_n \to l$, $k_n \to k$ and
$h_n = k_n l_n^{-1} \to kl^{-1}$. It follows that $H_K$ is also
compact.
\end{proof}

Let us clarify the relation between covering-space actions, discrete
actions, and proper actions. The following result is an adaptation of
\cite[Lemma~21.11]{Lee}.

\begin{lemma} 
\label{lm:discrete-vs-proper}
A continuous action of a discrete group $G$ on a smooth manifold $X$
is free and proper if and only if it is discrete.
\end{lemma}

\begin{proof}
Let $x \in X$ and $V$ be a neighborhood for $x$ with $\ovl{V}$ compact
in $X$. By Lemma~\ref{lm:equivalent-proper-actions} the set
$\set{g \in G : (\ovl{V}\.g) \cap \ovl{V} \neq \emptyset}$ is compact
and hence finite, since $G$ is discrete. Let us write it as
$\{e,g_1,\dots,g_m\}$. We may shrink $V$ to ensure that
$x \notin \ovl{V}\.g_1$; and by repeating that process at most
$m$~times, we can suppose that $x\.g_i^{-1} \notin \ovl{V}$ for
$i = 1,\dots,m$. The set
$U = V \less (\ovl{V}\.g_1 \cupycup \ovl{V}\.g_m)$ satisfies
$U \cap U\.g = \emptyset$ for all $g \neq e$ in~$G$; this means that
the action of~$G$ is a covering-space action.

By Lemma~\ref{lm:proper-action-characterization} the quotient space
$X/G$ has a Hausdorff topology. Let $x,y \in X$ have different orbits.
Since $\pi(x) \neq \pi(y)$ in~$X/G$, there are disjoint neighborhoods
$U'$ and $V'$ of these points in $X/G$. If $U = \pi^{-1}(U')$ and
$V = \pi^{-1}(V')$, then $z \in (U\.g) \cap V$ implies $z = z'\.g$ for
some $z'\in U$ . This contradicts the assumption that $U'$ and $V'$
are disjoint since $\pi(z) = \pi(z')$ and the action is discrete.

For the converse, we have already noted that every covering-space
action, and hence every discrete action, is free. If $(g_n)$ and
$(x_n)$ are sequences in $G$ and $X$ respectively, with $g_n \to g$
and $x_n\.g_n \to x'$ with $x$ and~$x'$ in different orbits, there are
neighborhoods $V$ of~$x$ and $V'$ of~$x'$ with disjoint orbits. For
large enough $n$ we find that $x_n \in V$ and $x_n\.g_n \in V'$,
contrary to hypothesis; therefore, $x$ and~$x'$ belong to the same
orbit. Let $g \in G$ with $x\.g = x'$ so as to have
$x_n\.g_ng^{-1} \to x$. If $U$ is a neighborhood of $x$ taken from the
covering-space action property, for $n$ large enough both $x_n$ and
$x_n\.g_ng^{-1}$ lie in~$U$, implying that
$U\.g_n g^{-1} \cap U \neq \emptyset$. Thus $g = g_n$ for every $n$
large enough. The result follows once more from
Lemma~\ref{lm:equivalent-proper-actions}.
\end{proof}

Theorem~5.8 in \cite{NovikovTaimanov} states that if the action of $G$
on $X$ is free and discrete, then one can introduce a smooth atlas on
$X/G$ making $X/G$ a smooth manifold for which the projection $X \to
X/G$ is a smooth map; this is a version of the Quotient Manifold
Theorem that will be taken up in
Section~\ref{sec:Husemoller-principal-bundles}. In the following
theorem, gleaned from Example~4.8 of \cite{DoCarmo}, the idea is to
observe that the restriction of the projection map $\pi \: X \to X/G$
to each neighborhood $U$ of the covering-space action is injective,
and that for any two such neighborhoods $U_1$ and $U_2$, the map
$\pi_1^{-1} \circ \pi_2$ is smooth. It provides a smooth structure on
the quotient space simply as a result of the action being a
covering-space action.

\begin{theorem} 
\label{thm:DoCarmo}
If the action of the discrete group $G$ on the smooth manifold $X$ is
a covering-space action, then there is a smooth structure on $X/G$
which makes the projection $X \to X/G$ a local diffeomorphism.
\end{theorem}

\begin{proof}
Let $x \in X$ and let $U$ be a open neighborhood of $x$ such that
$U \cap U\.g = \emptyset$ for $g \neq e$. Let
$\psi \: V\subseteq \bR^n \to U$ be a local parametrization at~$x$.
Since $\pi\circ\psi \: V \to X/G$ is injective, $\pi|_U$ is also
injective.
 
Consider the system of local parametrizations $\{(V,\pi\circ\psi)\}$
for $X/G$. It will be shown that
$(\pi\circ\psi_1)^{-1} \circ (\pi\circ\psi_2)$ is differentiable,
whenever
$(\pi\circ\psi_1)(V_1) \cap (\pi\circ\psi_2)(V_2) \neq \emptyset$.

Let $b \in (\pi\circ\psi_1)(V_1) \cap (\pi\circ\psi_2)(V_2)
\subseteq X/G$ and let
$z = \psi_2^{-1} \circ (\pi|_{\psi_2(V_2)})^{-1}(b) \in V_2$. Let
$W \subseteq V_2$ be a neighborhood of $z$ such that
$(\pi|_{\psi_2(V_2)} \circ \psi_2)(W)
\subseteq (\pi\circ\psi_1)(V_1) \cap (\pi\circ\psi_2)(V_2)$. In this
way,
$$
(\pi\circ\psi_1)^{-1} \circ (\pi\circ\psi_2) \bigr|_W 
= \psi_1^{-1} \circ (\pi|_{\psi_1(V_1)})^{-1}
\circ (\pi|_{\psi_2(V_2)}) \circ \psi_2.
$$

To see that
$(\pi|_{\psi_1(V_1)})^{-1} \circ (\pi|_{\psi_2(V_2)})$ is smooth at
$x_2 = (\pi|_{\psi_2(V_2)})^{-1}(b)$, first notice that $x_2$ and
$x_1 = (\pi|_{\psi_1(V_1)})^{-1} \circ (\pi|_{\psi_2(V_2)})(x_2)$ are
equivalent in~$X$, since both lie over~$b$, so there is a
$g \in G$ such that $x_2\.g = x_1$. It follows that the restriction
$(\pi|_{\psi_1(V_1)})^{-1} \circ (\pi|_{\psi_2(V_2)})|_{\psi_2(W)}$
coincides with $\al_g|_{\psi_2(W)}$. As a result, whatever property is
imposed on $\al_g$ (differentiable, $C^k$, smooth), it will be
inherited by $(\pi|_{\psi_1(V_1)})^{-1} \circ (\pi|_{\psi_2(V_2)})$
at~$x_2$. It immediately follows that the projection $\pi$ is smooth
under this smooth structure for~$X/G$.
\end{proof}

The previous theorem generalizes, for example, the construction of the
projective spaces from the action of $\bZ_2$ acting on the spheres, or
the torus from $\bZ^k$ acting on $\bR^k$ and the Klein bottle from
$\bZ_2$ acting on the $2$-torus: see \cite[p.~22]{DoCarmo} and
\cite[p.~143]{NovikovTaimanov}. Note that for a covering-space
action, each orbit is homeomorphic to the discrete group~$G$.

By Lemma~\ref{lm:discrete-vs-proper}, the action is proper; and by
Lemma~\ref{lm:proper-action-characterization}, the quotient $X/G$ is
a Hausdorff space. We conclude the following.

\begin{proposition} 
\label{pr:Novi-Taima}
If the action of the discrete group $G$ on the smooth manifold $X$ is
a discrete action, then $X/G$ is a Hausdorff space.
\qed
\end{proposition}

\begin{remark}
In Theorem~\ref{thm:DoCarmo} we have shown the existence of a smooth
structure on $X/G$. By Proposition~\ref{pr:Novi-Taima}, once we
require the action to be discrete, we know that $X/G$ is a Hausdorff
space. Because the projection is a local diffeomorphism, it follows
that $X/G$, like~$X$, is second countable.
\end{remark}

Because $G$ is discrete, $\tau \: X \x_G X \to G$ is trivially
continuous. If the action of $G$ on $X$ is discrete, $X \x_G X$ is
homeomorphic to $X \x G$. By Theorem~\ref{thm:DoCarmo} and
Proposition~\ref{pr:Novi-Taima}, one can form the bundle $(X,\pi,X/G)$.

\begin{definition} 
\label{def:free-covering-G-space}
A \textit{covering $G$-space} is a bundle $(X,\pi,X/G)$ with $G$ a
discrete group, $X$ a smooth manifold with a covering-space (hence
free) action. A bundle $(X,p,B)$ is called a
\textit{covering $G$-bundle}%
\footnote{%
This is not a standard nomenclature, we use it to emphasize the role
of the action being a covering-space action in the smooth setting,
instead of being proper in the topological one.}
if it is isomorphic to a covering $G$-space by an isomorphism
$(1,f) \: (X,\pi,X /G) \to (X,p,B)$, i.e., $p = f \circ \pi$ where
$f \: X /G \to B$ is a \textit{diffeomorphism}.

If the action is discrete, we get a \textit{discrete $G$-space} and a
\textit{discrete $G$-bundle}, respectively.
\end{definition}

By Theorem~\ref{thm:DoCarmo}, the base $B$ in any covering $G$-bundle
$(X,p,B)$ is locally diffeomorphic to its total space~$X$. By
Lemma~\ref{lm:discrete-vs-proper} if the action is discrete, then the
covering $G$-bundle has a free and proper (principal) action. By
Lemma~\ref{lm:continuous-free-and-proper-imply-homeomorphic-fibers},
every fiber is homeomorphic to the discrete group~$G$.

\medskip

Next, we address the principal bundle situation for the case of a
discrete group acting on a smooth manifold, and the corresponding
properties of the translation map.

Up to this point, we have been working with a discrete group $G$
acting continuously on a smooth manifold~$X$. If we ask for a
covering-space action, then $X/G$ has a smooth manifold structure.
Furthermore, the action will be free, so there will be a
well-defined translation map.

We now reformulate the definition of fiber bundle from
Section~\ref{sec:topological-bundles}, in a way suitable for the
smooth situation.

With $X$ and $B$ smooth manifolds, a \textit{(smooth) fiber bundle}
$(X,p,B,F)$ (whose \textit{fiber} is a smooth manifold~$F$) is a
smooth and surjective map $p \: X \to B$ such that at each point
of~$B$, there is a neighborhood $U$ and a diffeomorphism
$\psi \: U \x F \to p^{-1}(U)$ such that
$p \circ \psi \: U \x F \to U$ is the projection $\pr_1$ on the first
factor. The system $\{(U_i,\psi_i)\}$ is called a \textit{coordinate
representation} for the bundle.

The following result (\cite[vol.~I, p.~39]{Greubetal}) identifies
those sets that can be used as the total space of a fiber bundle over
a given manifold.

\begin{proposition} 
\label{pr:sets-as-bundles}
If $B$ and $F$ have smooth structures and $E$ is a set together with a
surjective map $p \: E \to B$ having the following properties:
\begin{enumerate}[label={\textup{(\roman*)}}]
\item 
there is an open covering $\{U_i\}$ of $B$ and a family $\{\psi_i\}$
of bijections $\psi_i \: U_i \x F \to p^{-1}(U_i)$;
\item 
$p \circ \psi_i(b,y) = b$ for every $b \in U_i$, $y \in F$;
\item 
the maps $\psi_{ij} \: (U_i \cap U_j) \x F \to (U_i \cap U_j) \x F$
defined by $\psi_{ij} = \psi_i^{-1} \circ \psi_j$ are diffeomorphisms;
\end{enumerate}
then there is exactly one smooth structure on~$E$ for which
$(E,p,B,F)$ is a fiber bundle with coordinate representation
$\{(U_i,\psi_i)\}$.

Furthermore, if $B$ and $F$ are Hausdorff, so also is~$E$.
\end{proposition}

\begin{proof} 
We first show uniqueness of the smooth structure. Let $\sA$ and
$\sB$ be atlases on $E$, both making the maps
$\psi \: U \x F \to p^{-1}(U)$ diffeomorphisms. Take
$(V_i,\al_i) \in \sA$, $(V_j,\bt_j) \in \sB$; $U \subseteq B$ a chart
domain with $V_i \cap V_j \cap p^{-1}(U) \neq \emptyset$; and
$(W,\eta)$ a local chart for $U \x F$. In this way,
$\eta \circ \psi^{-1} \circ \al_i^{-1}$ and
$\eta \circ \psi^{-1} \circ \bt_j^{-1}$ are diffeomorphisms on their
restricted domains, so
$\eta \circ \psi^{-1} \circ \al_i^{-1} \circ \bt_j$ is a
diffeomorphism. Therefore, each $\al_i^{-1} \circ \bt_j$ is a
diffeomorphism and $\sA \sim \sB$.

We next provide $E$ with a smooth structure. Let $(W,\eta)$ be a local
chart for $F$. By shrinking the $U_i$ if needed, we may consider each
$U_i$ to be a chart domain (with local chart $\phi_i$) in~$B$. Each
map $\psi_i \: U_i \x F \to p^{-1}(U_i) \subseteq E$ is a bijection.
Take its restriction $\wt{\psi_i} := \psi_i|_{U_i \x W}$ and consider
$\bigl( \wt{\psi_i}(U_i \x W), 
(\phi_i \x\eta)\circ \wt{\psi_i^{-1}} \bigr)$ as a candidate chart
for~$E$, with $\phi_i \x \eta \: 
U_i \x W \to (\phi_i \x \eta)(U_i \x W) \subseteq \bR^{\dim B+\dim F}$
a local diffeomorphism.

By construction, $(\phi_i \x \eta) \circ \wt{\psi_i^{-1}}
\circ ((\phi_j \x\eta') \circ \wt{\psi_j^{-1}})^{-1}
= (\phi_i \x \eta) \circ \wt{\psi_i^{-1}} \circ \wt{\psi_j}
\circ (\phi_j \x \eta')^{-1}
= (\phi_i \x \eta) \circ \wt{\psi_{ij}} \circ (\phi_j \x \eta')^{-1}$
is a diffeomorphism, the $\wt{\psi_{ij}}$ being the appropriate
restrictions of the diffeomorphisms $\psi_{ij}$. The topology of~$E$
is given by the arbitrary unions of finite intersections of the sets
$\wt{\psi_i}(U_i \x W)$. Note that the
$\psi_i \: U_i \x F \to p^{-1}(U_i)$ are (locally) smooth maps since
by taking local charts as declared above, each
$(\phi_j \x \eta) \circ \wt{\psi_j^{-1}} \circ \wt{\psi_i}
\circ (\phi_i \x \eta)^{-1}$ is smooth.

By~(ii), the restriction of $p$ to $p^{-1}(U_i)$ is
$\pr_1 \circ \psi_i^{-1} \: p^{-1}(U_i) \to U_i$. By taking local
charts as before, it follows that
$\phi_i \circ p \circ \wt{\psi_i} \circ (\phi_i \x \eta)^{-1}
= \phi_i \circ \pr_1 \circ (\phi_i \x \eta)^{-1}$ which is a smooth
map. Hence $p$~is smooth on $E$ and $\{(U_i,\psi_i)\}$ is a local
decomposition for~$p$.

If $x \neq x'$ in $E$ with $p(x) \neq p(x')$, we can consider disjoint
chart domains $U_i$ and $U_j$ around $p(x)$ and $p(x')$ -- because $B$
is Hausdorff -- to construct disjoint neighborhoods
$\wt{\psi_i}(U_i \x W)$ for~$x$ and $\wt{\psi_j}(U_j \x W')$ for~$x'$
in~$E$. On the other hand, if $p(x) = p(x')$, we can shrink $U_i$ and
$U_j$ to $U_i \cap U_j$ to get $x,x' \in p^{-1}(U_i) \approx U_i \x F$
by~(i). We can then take disjoint chart domains $W$ and $W'$ in $F$ --
because $F$ is Hausdorff -- to construct disjoint open neighborhoods
$\wt{\psi_i}(U_i \x W)$ for~$x$ and $\wt{\psi_i}(U_i \x W')$ for~$x'$.
Thereby $E$ is seen to be a Hausdorff space.
\end{proof}

In the previous theorem, if one asks that $B$ and $F$ be paracompact
(or have a countable basis), the same will hold for~$E$. That is to
say, if $B$ and $F$ are manifolds, then $E$ will be a manifold, too.

\medskip

We know that there is a homeomorphism
$G \approx \{b\} \x G \to p^{-1}(b)$ of discrete spaces, which is
trivially a diffeomorphism. More generally, the fiber can be taken to
be a Lie group~$G$. In that case, a (smooth) \textit{principal fiber
$G$-bundle} is a smooth fiber bundle $(X,p,B,G)$ where $X$ carries a
smooth (right) action of $G$, together with a coordinate
representation $\{(U_i,\psi_i)\}$ such that
$\psi_i(b,gh) = \psi_i(b,g)\.h$ for $b \in U_i$ and $g,h \in G$, as in
\eqref{eq:principal-condition}.

The next natural question is, are covering-space $G$-bundles with a
smooth action actually smooth principal $G$-bundles? That is to say,
do they have coordinate representations $\{(U_i,\psi_i)\}$ satisfying
\eqref{eq:principal-condition}? In order to obtain a Hausdorff and
paracompact base, we must demand that the action be discrete.

\begin{theorem} 
\label{thm:covering-principal-is-principal-discrete-case}
Every discrete $G$-bundle $(X,p,B)$ with a smooth discrete action is a
smooth principal fiber $G$-bundle.
\end{theorem}

\begin{proof}
Given the isomorphism $(1,f)$ of
Definition~\ref{def:free-covering-G-space} with $(X,\pi,X/G)$, it is
enough to show that $(X,\pi,X/G)$ has the desired coordinate
representation.

Since the action is a covering-space action, each $x \in X$ has an
open neighborhood $U_x$ such that $U_x \cap U_x\.g = \emptyset$ for
all $g \neq e$ in~$G$. By the local diffeomorphism established in
Theorem~\ref{thm:DoCarmo}, we can shrink $U_x$ enough to be
diffeomorphic to $V_x = \pi(U_x)$, which is open in (the quotient
topology of) $X/G$; and so that $X/G$ is covered by all such open
sets~$\{V_x\}$.

Because the group $G$ is discrete, $\pi^{-1}(V_x)$ is the disjoint
union $\biguplus_{g\in G} U_x\.g$. Consider the homeomorphism
$\theta_x \: \biguplus_{g\in G} U_x\.g \to U_x \x G$ with
$\theta_x(x'\.g) = (x',g)$. On this domain we can consider the
(disjoint) charts $\set{U_x\.g : g\in G}$, with each $U_x\.g$
diffeomorphic to $U_x$ since the action is smooth. On the codomain,
since $G$ is discrete, the charts can be taken as
$\set{U_x \x \{g\} : g \in G}$, which are also diffeomorphic to $U_x$
and~$V_x$, as noted before. In this way $\theta_x$ is a smooth map;
its inverse $\theta_x^{-1}$ is the restriction of the action of $G$ on
$X$ to~$U_x$.

Define maps $\psi_x \: V_x \x G \to \pi^{-1}(V_x)$ by using the
diffeomorphism between $U_x$ and $V_x = \pi(U_x)$, namely
$\psi_x(\pi(x'),g) := x'\.g = \theta_x^{-1}(x',g)$. Thus $\psi_x$ is
also a diffeomorphism, which furthermore satisfies
$\pi(\psi_x(\pi(x'),g)) = \pi(x'\.g) = (x'\.g)G = x'\.G$. Finally note
that 
$\psi_x(\pi(x'), gh) = x'\.gh = (x'\.g)\.h = \psi_x(\pi(x'),g)\.h$.

The Hausdorff and paracompact structures on the base $B$ follow from
the discreteness of the action, which makes the quotient space $X/G$
Hausdorff and paracompact. The result is straightforward from the
isomorphism $(1,f)$.
\end{proof}

By Lemma~\ref{lm:proper-action-characterization} we know that a free
action of a topological group $G$ on a topological space is proper if
and only if the quotient space has a Hausdorff topology and the
translation map $\tau$ is continuous. By
Lemma~\ref{lm:discrete-vs-proper}, a continuous action of a discrete
group $G$ on a smooth manifold $X$ is free and proper if and only if
it is discrete. From Theorem~\ref{thm:DoCarmo} we then obtain the
following conclusion.

\begin{corollary} 
\label{cor:quotient-manifold-and-tau-smooth-1}
If the action of a discrete group $G$ on a smooth manifold $X$ is
smooth and discrete, then $X/G$ has the structure of a smooth manifold
locally diffeomorphic to $X$ and the translation map
$\tau \: X \x_G X \to G$ is continuous.
\qed
\end{corollary}

A continuous translation map is equivalent to a homeomorphism 
$\ga \: X \x G \stackrel{\approx}{\to} X \x_G X$. If we are to provide
$X \x_G X$ with a smooth structure in order to obtain a differentiable
map $\tau \: X \x_G X \to G$ with $G$ discrete, that will require, for
any chart $(U,\vf)$ on $X \x_G X$ with $\vf \: U \to V$ open
in~$\bR^k$, that the map $\tau \circ \vf^{-1} \: V \to \{g\}$ be
differentiable. Since this is actually a constant map, any smooth
structure on $X \x_G X$ will suffice. By Lemma~\ref{lm:DuKo} and since
$X/G$ is in particular Hausdorff under a discrete action, $X \x_G X$
is a closed subset of $X \x X$. We next show that for a discrete group
acting on a smooth manifold, there is a preferred smooth structure on
$X \x_G X$ related to $X \x G$.

\begin{corollary} 
\label{cor:quotient-manifold-and-tau-smooth-2}
If the action of a discrete group $G$ on a smooth manifold $X$ is
smooth and discrete, then $X \x_G X$ has a smooth manifold structure
diffeomorphic to $X \x G$, and the translation map is trivially
smooth.
\end{corollary}

\begin{proof}
By Lemma~\ref{lm:XxG-homeomorphism}, if $\tau$ is
continuous, we know that $X \x_G X$ is homeomorphic to $X \x G$ via
the
inverse $\ga^{-1}$ of the canonical map. Since $G$ has the discrete
topology, $X \x G$ has trivially a manifold structure. If $X$ is
Hausdorff and paracompact, so too will be $X \x G$ and $X \x_G X$.
Indeed, if $\set{(U_i,\phi_i) : i \in I}$ is an atlas for $X$, then
$\set{(U_i \x \{g\},\phi_{i,g}) : i \in I, \ g \in G}$, with
$\phi_{i,g}(x,g) = \phi_i(x)$, is an atlas for $X \x G$. Our candidate
atlas on $X \x_G X$ is
$\set{(\ovl{U}_{i,g}, \bar\phi_{i,g}) : {i \in I},\ g \in G}$ with
$\ovl{U}_{i,g} = \set{(x,x\.g) : x \in U_i, \ g \in G}$ (open in
$X \x_G X$ and homeomorphic to $U_i \x \{g\}$) and
$\bar\phi_{i,g}(x,g\.x) := \phi_i(x)$. Note that
$\ovl{U}_{i,g} \cap \ovl{U}_{j,h} \neq \emptyset$ if and only if
$g = h$ and $U_i \cap U_j \neq \emptyset$. In that intersection,
$\bar\phi_{i,g} \circ \bar\phi_{j,g}^{-1} = \phi_i \circ \phi_j^{-1}$,
so this is indeed an atlas for $X \x_G X$. With this structure,
$X \x_G X$ is diffeomorphic to $X \x G$ by taking $(x,x\.g)$ to
$(x,g)$; and now $\tau \: X \x_G X \to G$ composes this with the
projection $\pr_2$ on $X \x G$, and hence it is a smooth map.
\end{proof}

The next concept is independent of whether or not $G$ is discrete, and
will be used later in Section~\ref{sec:Husemoller-principal-bundles}
where $G$ is a Lie group. Consider a principal (right) $G$-bundle
$(X,p,B)$ and a left $G$-space $F$. On $X \x F$ a $G$-structure is
given by putting $(x,y)\.g := (x\.g, g^{-1}\.y)$. (The same $\cdot$ is
used to denote different actions of~$G$.) Denote the quotient space
obtained from this action by $(X \x F)/G$. Define the map
$p_F \: (X \x F)/G \to B : (x,y)\.G \mapsto p(x)$. This map is well
defined thanks to the freedom of the action in a principal $G$-bundle.
\[
\vcenter{
\xymatrix{\ar @{} [dr]
X \ar@{<->}[r]^1 \ar[d]_\pi & X \ar[d]^p
\\
X/G \ar[r]^f & B}}
\qquad  \implies  \qquad
\vcenter{
\xymatrix{\ar @{} [dr]
X \x F \ar[r]^(.43){\pi_\x} \ar[d]_{p \x 1} & (X \x F)/G \ar[d]^{p_F}
\\
B \x F \ar[r]^(.55){\pr_1} & B.}}
\]

The quadruple $((X \x F)/G,p_F,B,F)$ is called a \textit{fiber bundle
with fiber} $F$, \textit{associated} to the principal $G$-bundle
$(X,p,B)$.

We shall show in Corollary~\ref{cor:diff-fiber-covering-space} that
for the smooth fiber bundle with fiber~$F$ associated to the
covering-space principal $G$-bundle $(X,p,B)$, each $p_F^{-1}(b)$ is
diffeomorphic to $F$. The next question to be addressed will be: is an
associated smooth fiber bundle with fiber $F$ actually a fiber bundle?
This is dealt with in
Proposition~\ref{pr:proper-discont-fiber-bundle}. A trivial but useful
observation is contained in the following lemma.

\begin{lemma} 
\label{lm:covering-space-implies-covering-space}
If the action of $G$ on $X$ is discrete, and both actions of $G$ on
$X$ and~$F$ are smooth, then the action of $G$ on $X \x F$ is also a
discrete smooth action. Thus, the quotient space $(X \x F)/G$ has a
smooth manifold structure locally diffeomorphic to $X \x F$ and the
map $p_F \: (X \x F)/G \to B$ is smooth.
\end{lemma}

\begin{proof}
Smoothness follows immediately. For the action of $G$ on $X \x F$
given by $(x,y)\.g = (x\.g,g^{-1}\.y)$ to be a covering-space action,
we require that each $(x,y) \in X \x F$ has a neighborhood
$W \subset X \x F$ such that $W \cap W\.g = \emptyset$ for all
$g \neq e$. By hypothesis, there is an open neighborhood $U$ of $x$ in
$X$ such that $U\.g \cap U = \emptyset$ for all $g \in G \less \{e\}$.
Consider $W = U \x F$ which is open in $X \x F$. If $(x',y')$ belongs
to $W\.g \cap W \subset (U\.g \x F) \cap (U \x F)$ for some $g \in G$,
then $x' \in U\.g \cap U$, which forces $g = e$.

Next, let $(x,y)$ and $(x',y')$ in $X \x F$ have different orbits. We
must show that there are neighborhoods $W$ and $W'$ of these points
whose orbits are disjoint. By hypothesis, there are neighborhoods $U$
of~$x$ and $U'$ of~$x'$ whose orbits in~$X$ are disjoint. So it is
enough to consider $W = U \x F$ and $W' = U' \x F$ to conclude that
the action of $G$ on $X \x F$ is discrete.

The result follows from Theorem~\ref{thm:DoCarmo} and
Proposition~\ref{pr:Novi-Taima}. The smoothness of $p_F$ follows when
one considers it as the factorization of the composition
$X \x F \stackrel{\pi_\x}{\longrightarrow} X \stackrel{p}{\to} B$,
using that the projection $\pi_\x \: X \x F \to (X \x F)/G$ is a local
diffeomorphism.
\end{proof}

\begin{corollary} 
\label{cor:diff-fiber-covering-space}
If the action of $G$ on $X$ is a discrete smooth action and
$((X \x F)/G,p_F,B,F)$ is a smooth fiber bundle associated to the
covering principal $G$-bundle $(X,p,B)$, then for each $b \in B$ the
fiber $F$ is diffeomorphic to $p_F^{-1}(b)$.
\end{corollary}

\begin{proof}
First note that 
$p_F^{-1}(b) = \set{(x,y)\.G : x \in p^{-1}(b), \ y \in F}$. Let
$p(x_0) = b$ for some $x_0 \in X$ and consider the map
$k_{x_0} \: p^{-1}(b) \x F \to F : k_{x_0}(x,y) = \tau(x_0,x)\.y$. By
Corollary~\ref{cor:quotient-manifold-and-tau-smooth-2}, $k_{x_0}$ is
smooth as a composition of smooth maps. Since
$k_{x_0}(x\.g, g^{-1}\.y) = \tau(x_0,x\.g) g^{-1}\.y = \tau(x_0,x)\.y
= k_{x_0}(x,y)$, one can factor through the quotient map
$X \x F \to (X \x F)/G$ to get a smooth map
$k'_{x_0} \: p_F^{-1}(b) \to F 
: (x,y)\.G = (x_0\.\tau(x_0,x),y)\.G \mapsto \tau(x_0,x)\.y$. 

The map $f_{x_0} \: F \to p_F^{-1}(b) \subseteq (X \x F)/G
: y \mapsto (x_0,y)\.G$ factors as 
$$
F \to X \x F \to (X \x F)/G : y \mapsto (x_0,y) \mapsto (x_0,y)\.G.
$$
By
Lemma~\ref{lm:covering-space-implies-covering-space}, $f_{x_0}$ is a
smooth map as a restriction on its codomain of a smooth map. The
composition of smooth maps
$f_{x_0} \circ k'_{x_0} \: p_F^{-1}(b) \to p_F^{-1}(b)$ satisfies
$$
(x,y)\.G \mapsto \tau(x_0, x)\.y 
\mapsto (x_0, \tau(x_0,x)\.y)\.G = (x,y)\.G,
$$
which is the identity map; and the result follows.
\end{proof}

\begin{proposition} 
\label{pr:proper-discont-fiber-bundle}
The smooth fiber bundle $((X \x F)/G,p_F,B,F)$ with fiber $F$,
associated to the discrete%
\footnote{A covering principal bundle with a discrete and smooth
action.}
principal $G$-bundle $(X,p,B)$, is indeed a fiber bundle.
\end{proposition}

\begin{proof}
Consider the following diagrams; the one on the right shows open sets
in the corresponding spaces on the left. The aim is to show that there
is a local product structure on $((X \x F)/G,p_F,B,F)$, as follows:
\[
\xymatrix{\ar @{} [dr]
 X \ar@{<->}[r]^1 \ar[d]_{\pi} 
& X\ar[d]^p 
& X \x F \ar[l]_{\pr_1} \ar[r]^(.55){\pi_\x} \ar[d]_{p \x 1}  
& (X \x F)/G \ar[d]^{p_F}
\\
 X/G \ar[r]^f 
& B
& B \x F \ar[r]^(.55){\pr_1} 
& B,}
\quad
\xymatrix{\ar @{} [dr]
 U_x \ar@{<->}[r]^1 \ar[d]_{\pi} 
& U_x\ar[d]^p 
& U_{(x,y)} \ar[l]_{\pr_1} \ar[r]^(.55){\pi_\x} \ar[d]_{p \x 1}  
& W_x \ar[d]^{p_F}
\\
\pi(U_x) \ar[r]^f 
& V_{(x,y)}
& U_x \x F \ar[r]^(.55){\pr_1} 
& V_{(x,y)}.}
\]
Since the action of $G$ on $X \x F$ is a covering-space action, each
$(x,y)\in X \x F$ has an open neighborhood $U_{(x,y)}$ such that
$U_{(x,y)} \cap U_{(x,y)}\.g = \emptyset$ for all $g \neq e$. Since
the action of $G$ on $X$ is also a covering-space action, every
$x \in X$ has an open neighborhood $U_x$ in $X$ such that
$U_x \cap U_x\.g = \emptyset$ for all $g \neq e$. By the proof of
Lemma~\ref{lm:covering-space-implies-covering-space}, one can assume
that $U_x \x F = U_{(x,y)}$. Thus, both
$f \circ \pi|_{\pr_1(U_{(x,y)})} = p|_{\pr_1(U_{(x,y)})}$ and 
$\pi_\x|_{U_{(x,y)}}$ are injective.

Evidently, $W_{(x,y)} = \pi_\x(U_{(x,y)})$ is open in (the quotient
topology of) $(X \x F)/G$. Also, since $p \x 1 = (f \circ \pi) \x 1$
is an open map, it follows that
$V_{(x,y)} = \pr_1 \circ (p \x 1)(U_{(x,y)})$ is open in $B$, and
$B$ is covered by such open sets. Furthermore, if
$(x',y') \in U_{(x,y)}$ then $p_F(\pi_\x(x',y')) = p_F((x',y')\.G)
= p(x') = \pr_1(p \x1)(x',y')$, and so $p_F(W_{(x,y)}) = V_{(x,y)}$.
One seeks to define diffeomorphisms 
$\psi_{(x,y)} \: V_{(x,y)}  \x F \to p_F^{-1}(V_{(x,y)})$ such that 
$p_F \circ \psi_{(x,y)} = \pr_1$. For each $b \in V_{(x,y)}$ there is
a unique $x' \in U_x$ with $p(x') = b$. Define
$\psi_{(x,y)}(b,y') := (x',y')\.G$, for $y' \in F$. It follows that
$p_F \circ \psi_{(x,y)}(b,y') = p_F((x',y')\.G) = p(x') = b$. Assume 
that $(x',y')\.G = \psi_{(x,y)}(b,y') = \psi_{(x,y)}(b',y'')
= (x'',y'')\.G$; then $b = b'$ on applying $p_F$. Since
$(x',y') = (x''\.g,g^{-1}\.y'')$ for some $g \in G$, $x'$ and $x''$
belong to the same orbit. It follows from the uniqueness of~$x'$ that
$x' = x''$, so $g = e$ and thus $y' = y''$. As a result, each
$\psi_{(x,y)}$ is an injective map.

To verify that each $\psi_{(x,y)}$ is surjective, consider 
$(x'',y'')\.G \in p_F^{-1}(V_{(x,y)})$ and let 
$b = p_F((x'',y'')\.G) = p(x'') \in V_{(x,y)}$. There is a unique
$x' \in U_x$ such that $p(x') = b$, and $x' = x''\.g$ for some
$g \in G$. By taking $(b,g\.y'') \in V_{(x,y)} \x F$ we ensure that
$\psi_{(x,y)}(b,g\.y'') = (x',g\.y'')\.G = (x'\.g,y'')\.G
= (x'',y'')\.G$. From this calculation, one sees that
$\psi_{(x,y)}^{-1}$ is given by the relation
\begin{align*}
& (x'',y'')\.G \mapsto \bigl( b,\tau(x'',x')\.y'' \bigr) 
= \bigl( p(x''\.\tau(x'',x')),\tau(x'',x')\.y'' \bigr)
\\
&\quad = \bigl( 
p\bigl( x''\.\tau(x'',(p|_{U_x}^{-1})(p_F((x'',y'')\.G)) \bigr),
\tau\bigl( x'',(p|_{U_x}^{-1})(p_F((x'',y'')\.G)) \bigr)\.y'' \bigr),
\end{align*}
which makes it a smooth map.  Thus also, $\psi_{(x,y)}(b,y')
= (x',y')\.G = \bigl( (p|_{U_x})^{-1}(b),y' \bigr)\.G$, which shows
that $\psi_{(x,y)}$ is smooth, as required.
\end{proof}


\section{A Lie group acting on a smooth manifold} 
\label{sec:Husemoller-principal-bundles}

We now consider a Lie group $G$ acting on a smooth manifold $X$. In
this case the map $x \mapsto x\.g$ is smooth, for all $g \in G$, and
thus is a diffeomorphism. The aim is to characterize the actions
behind a principal bundle in this situation. The simplest case was
stated in Proposition~\ref{pr:action-of-G-on-H-is-proper-free}: the
action of a Lie group on a closed subgroup is proper and free.

Given that $X/G$ is not in general a manifold, nor even a Hausdorff
space, an extra request must be placed on the action itself. The first
item in this section is to review the Quotient Manifold Theorem, which
provides a manifold structure on~$X/G$. In
\cite[Thm.~1.11.4]{DuistermaatKolk}, it is established that the orbit
space of a proper and free $C^k$-action of a Lie group on a
$C^k$-manifold has a $C^k$-manifold structure. In
\cite[Thm.~21.10]{Lee} and in \cite[Thm.~5.119]{LeeJeffrey} it is
shown that if a Lie group $G$ acts smoothly, freely and properly on a
smooth manifold $X$, then the orbit space $X/G$ is a smooth manifold.
See also \cite[p.~136]{Greubetal}.

\begin{remark}
As pointed out in Section~\ref{sec:topological-bundles}, the
translation function is defined on the bundle $(X,\pi,X/G)$. Given any
other $G$-bundle $(X,p,B)$, each fiber $p^{-1}(p(x))$ through
$x \in X$ will equal the orbit $x\.G$. It is thereby possible to
define a translation map for the $G$-bundle $(X,p,B)$ acting on a pair
of points belonging to the same fiber. Once again, to study a
principal bundle from its so-defined translation map, it is enough to
consider only the case $(X,\pi,X/G)$.
\end{remark}

The first step is to show that on every principal bundle the action is
proper and free. All our base spaces, being smooth manifolds, are
Hausdorff, so we do not worry about the quasi situation as in
Section~\ref{sec:topological-bundles}.

\begin{proposition} 
\label{pr:principal-bundle-action-proper-free}
Let $(X,p,B,G)$ be a (smooth) principal $G$-bundle. The action of $G$
on~$X$ is proper and free.
\end{proposition}

\begin{proof}
Since the bundle is locally a product $U \x G$, the $G$ action is
free. Indeed, there is a coordinate representation $\{(U_i, \psi_i)\}$
and a right action of $G$ on $X$ such that for every
$z \in U_i \subseteq B$ and $g,h \in G$, 
$\psi_i(z,gh) = \psi_i (z,g)\.h$. Thus,
$\psi_i (z,g) = \psi_i (z,g)\.h = \psi_i(z,gh)$ implies $h = e$, since
$\psi_i$ is a diffeomorphism.
 
Let $K$ be a compact subset of $X$. As in the proof of
Proposition~\ref{pr:action-of-G-on-H-is-proper-free}, let $(g_n)$ be a
sequence in $G_K = \set{g\in G : (K\.g) \cap K \neq \emptyset}$. There
are two sequences $(x_n)$, $(y_n) \subset K$ with $x_n\.g_n = y_n$ for
all~$n$. Since $K$ is compact in $X$, by passing to appropriate
subsequences we can assume there are $x,y \in K$ with $x_n \to x$,
$y_n \to y$; thus, $p(x_n) = p(x_n\.g_n) = p(y_n)$ for all~$n$, so
$p(x) = p(y)$ and hence $x\.g = y$, for some $g \in G$. For $n$ large
enough, the points $x_n,y_n$ belong to a single chart domain $V$ at
$x\.g = y$, which is diffeomorphic to $U \x G$, with $U$ a chart
domain at $\pi(x) \in B$. Let us call this diffeomorphism
$\psi \: U \x G \to V$.

Observe that if $\psi(v_n,k_n) = y_n = x_n\.g_n = \psi(u_n,h_n)\.g_n
= \psi(u_n,h_ng_n)$, for some sequences $(u_n)$, $(v_n) \subset U$,
and $(k_n)$, $(h_n) \subset G$, then $v_n = u_n$ and
$k_n h_n^{-1} = g_n$, for all~$n$. Since
$(u_n,h_n) = \psi^{-1}(x_n)$ is a convergent sequence in $U \x G$,
there is some $h \in G$ with $h_n \to h$. Similarly, since
$(u_n,h_ng_n)$ converges in $U \x G$, there is some $k \in G$ with
$h_n g_n \to k$; and thus $g_n \to h^{-1}k$. Hence $G_K$ is compact;
properness of the action follows by
Lemma~\ref{lm:equivalent-proper-actions}.
\end{proof}

We shall see that these properties characterize principal bundles
over smooth manifolds.

\begin{remark}
In \cite[vol.~II, p.~135]{Greubetal}, an action is called
\textit{proper} if for all compact subsets $A, B \subset X$, the
subset $H = \set{g \in G : A \cap B g \neq \emptyset}$ of~$G$ is
compact (see Lemma~\ref{lm:equivalent-proper-actions}). In that
reference, for a proper \textit{smooth} action, the orbits are closed
submanifolds and $X/G$ is Hausdorff and locally compact. Furthermore,
an exercise (as a result attributed to Andrew M. Gleason) asks to show
that if the action is proper and free then $X/G$ possesses a unique
smooth structure for which the quotient map is a submersion. This
result is our Proposition~\ref{pr:orbits-proper-actions}.
\end{remark}

\begin{proposition} 
\label{pr:orbits-proper-actions}
Let $G$ be a Lie group acting properly (hence continuously) on a
smooth manifold $X$. For each $x \in X$, the orbit map
$\al_x \: G \to X : g \mapsto x\.g$ is a proper map and so the orbit
$x\.G$ is closed in $X$. If the action is moreover free and smooth,
then the orbit map is a smooth embedding, the orbit is an embedded
submanifold, and the inclusion $x\.G \hookto X$ is a proper map.
\end{proposition}

\begin{proof}
Since the action $X \x G \to X$ is continuous, the orbit map $\al_x$
is continuous, too. By Lemma~\ref{lm:proper-map}, since $G$ is a Lie
group and $X$ a smooth manifold, to show that the orbit map is proper,
it is enough to verify that $K_x = \set{g\in G : x\.g \in K}$ is
compact in $G$, for every compact set $K \subseteq X$ and a given
$x \in X$. If $(g_n) \subset K_x$ is a sequence, then
$(x\.g_n) \subset K$ has a convergent subsequence, say
$x\.g_{n_k} \to y \in K$. By Lemma~\ref{lm:equivalent-proper-actions},
there is a new subsequence of $(g_{n_k})$ convergent in $G$, and hence
the orbit map is proper. Every proper map is a closed map, so each
orbit is closed in $X$. Now $\al_x$, being continuous and closed, is a
homeomorphism from $G$ to the orbit $x\.G$. As a result, the inclusion
$x\.G \hookto X$ is a proper map.

Now assume that the action is free and smooth. Because it is free,
$\al_x$ is injective. For a given $x \in X$, the orbit map $\al_x$ can
be written as a composition
$G \approx \{x\} \x G \hookto X \x G \to X$, where the first map is a
diffeomorphism and the last is the action itself. As a result, $\al_x$
is a smooth map, for every $x \in X$. Because the action is smooth,
each map $\al^g \: x \mapsto x\.g$ is a diffeomorphism on~$X$. Since
$\al_x \circ \mu_h = \al^h \circ \al_x$, with $\mu_h(g) = gh$, it
follows that $(d\al_x)_{gh} (d\mu_h)_g = (d\al^h)_{x\.g} (d\al_x)_{g}$
for every $g \in G$, and so the orbit map is of constant rank since
$(d\mu_h)_g$ and $(d\al^h)_{x\.g}$ are isomorphisms. Thus each $\al_x$
is a smooth immersion and a homeomorphism as well, i.e, it is a smooth
embedding; and the result follows.
\end{proof}

The second part of the next result states a smoothness property of the
translation map, that somehow replaces its continuity from the
topological case. One would like to consider translation functions
$\tau \: X \x_G X \to G$ with some sort of smoothness, which amounts
to providing $X \x_G X$ with an appropriate smooth manifold structure.
As suggested in Corollary~\ref{cor:diffeomorphic-orbits}, an
alternative approach is to consider the map $\tau_x \: x\.G \to G$
given by $x' = x\.g \mapsto \tau_x(x') := \tau(x,x') = g$. An
advantage is that by Proposition~\ref{pr:orbits-proper-actions} the
orbits are embedded submanifolds diffeomorphic to $G$.

\begin{corollary} 
\label{cor:diffeomorphic-orbits}
Let $G$ be a Lie group acting properly, freely and smoothly on a
smooth manifold $X$. Each orbit in $X$ is diffeomorphic to $G$ and the
restriction of the translation map to each orbit,
$\tau_x \: x\.G \to G : x\.g \mapsto \tau(x,x\.g) = g$, is smooth.
\end{corollary}

\begin{proof}
To see that each orbit is diffeomorphic to $G$, it suffices to
consider $\al_x \: G \to x\.G$ which is a surjective immersion of
constant rank. Its inverse map
$\al_x^{-1} \: x\.G \to G : x\.g \mapsto g$ is smooth and is precisely
the restriction $\tau_x$ of the translation map to the orbit $x\.G$.
\end{proof}

From Proposition~\ref{pr:principal-bundle-action-proper-free}, the
smoothness of $\tau_x$ holds for any (smooth) principal
$G$-bundle $(X,p,B,G)$, through the appropriate identifications given
by a diffeomorphism $f \: X/G \to B$.

\medskip

Next, we want to verify that the quotient $X/G$ is actually a smooth
manifold.

If $G$ is a $k$-dimensional Lie group acting on a $n$-dimensional
smooth manifold $X$, a chart $(U,\vf)$ centered at $x \in X$ is
\textit{adapted to the action of~$G$} if $\vf(U) = A \x B$ where
$A \subseteq \bR^k$ and $B \subseteq \bR^{n-k}$ are open subsets; with
coordinate functions $(x^1,\dots,x^k,y^1,\dots,y^{n-k})$ such that
each $G$-orbit is either disjoint from~$U$, or has intersection
with~$U$ being a single slice of the form
$(y^{k+1},\dots,y^{n-k}) = (c^{k+1},\dots,c^n)$, for some
constants~$c^i$.

\begin{remark}
By shrinking the sets $A$ and $B$, they can be taken as open cubes in
$\bR^k$ and $\bR^{n-k}$. Also, by taking
$\vec{c} = (c^{k+1},\dots,c^n) \in B$ fixed,
$\vf^{-1}(A \x \vec{c}) = U \cap \sO_{\vf^{-1}(\vec{0},\vec{c})}$ is
the intersection of $U$ with the orbit through the point
$\vf^{-1}(\vec{0},\vec{c})$ -- or through any other point
$\vf^{-1}(\vec{a},\vec{c})$, if $\vec{a} \in A$.
\end{remark}

The following result quotes both the second part of the proof of
\cite[Thm.~21.10]{Lee} and \cite[Thm.~5.115]{LeeJeffrey}. There is
also another version in the proof of
\cite[Thm.~1.11.4]{DuistermaatKolk}. Here we switch back to left
actions during the proof of the following four results.

\begin{lemma} 
\label{lm:adapted-charts}
Let $G$ be a $k$-dimensional Lie group acting smoothly, freely and
properly on an $n$-dimensional smooth manifold $X$. Then at each
$x \in X$ there is a chart $(U,\vf)$ adapted to the action of~$G$.
\end{lemma}

\begin{proof}
By Proposition~\ref{pr:orbits-proper-actions}, each orbit is an
embedded submanifold of $X$. Hence, there is a chart $(U',\vf')$ of
$X$, centered at~$x$ and such that $G\.x \cap U'$ has the form
$y^{k+1} =\cdots= y^n = 0$. Let $S$ be the $(n-k)$-dimensional
embedded submanifold (with a unique chart) defined as
$\set{u \in U' : x^1(u) =\cdots= x^k(u) = 0}$ with coordinate chart
$$
S = \vf'^{-1}(\vf'(U') \cap (0^k \x \bR^{n-k}))
\to \pi_{n-k}(\vf'(S)) \subseteq \bR^{n-k}
$$ 
and coordinates $(y^{k+1},\dots,y^n)$, the last $(n-k)$ coordinates
of~$\vf'$. Thus, $T_x X = T_x(G\.x) \oplus T_x S$.

Let $\al|_S \: G \x S \to X$ be the restriction of the action of~$G$,
$i_x \: G \hookto G \x S : g \mapsto (g,x)$ and
$j_e \: S \hookto G \x S : s \mapsto (e,s)$ be the two inclusions,
which are smooth embeddings. With $\iota \: S \hookto X$ the inclusion
map, we see that $\al_x = \al|_S \circ i_x$ and
$\al|_S \circ j_e = \iota$, and so
$(d\al_x)_g = (d\al|_S)_{(g,x)} \circ (di_x)_g$ and
$(d\al|_S)_{(e,s)} \circ (dj_e)_s = (d\iota)_s$. Note that
$(d\al_x)_e(T_e G) = T_x(G\.x)$ since $\al_x$ is an embedding,
implying that the image of $(d\al|_S)_{(e,x)}$ must include
$T_x (G\.x)$. Similarly, the image of $(d\al|_S)_{(e,x)}$ must include
$T_x S$, and so $(d\al|_S)_{(e,x)} \: T_{(e,x)}(G \x S) \to T_x X$ is
surjective. Because $\dim(T_{(e,x)}(G \x S)) = \dim(T_x X)$, it
follows that $(d\al|_S)_{(e,x)}$ is bijective.

By the Inverse Function Theorem, there is an open neighborhood of
$(e,x)$, that can be taken small enough to satisfy three properties:
\begin{itemize}
\item
it is of the form $A' \x B'$ with $A'$ open in $G$ and $B'$ open
in~$S$;
\item
$\al|_S(A' \x B') \subseteq U'$; and 
\item 
$(\al|_S)|_{A'\x B'}$ is a diffeomorphism. 
\end{itemize}
Consider diffeomorphisms $a \: (-1,1)^k = I^k \to A'$ and 
$b \: I^{n-k} \to B'$ with $a^{-1}(e) = 0$, $b^{-1}(x) = 0$. The
inverse of the diffeomorphism
$\al|_S \circ (a \x b) \: I^k \x I^{n-k} \to \al|_S(A' \x B')$ defines
a chart with domain $\al|_S (A' \x B')$. 

Now $B'$ can be taken small enough so that its intersection with each
orbit is empty or a single point: otherwise, then there would be a
sequence of open sets $(B'_n)$ with compact closures such that
$\ovl{B'_{n+1}} \subseteq B'_n$, each with two different points
$x_n,x'_n\in B'_n$ satisfying $g_n\.x_n = x'_n$ for some $g_n \in G$.
The sequences $(g_n\.x_n)$ and $(x'_n)$ both converge to~$x$; by
Proposition~\ref{lm:equivalent-proper-actions} and since the action is
proper, there is a subsequence $(g_{n_k})$ convergent to some
$g \in G$. In this way $g\.x \gets g_{n_k}\.x_{n_k} = x'_{n_k} \to x$,
and so $g = e$ because the action is free. This would imply that
$g_{n_k} \in A$ for $k$ large enough, and so
$\al|_S(g_{n_k},x_{n_k}) = \al^{g_{n_k}}(x_{n_k}) = x'_{n_k}
= \al^e(x_{n_k}) = \al|_S(e,x'_{n_k})$, contradicting the injectivity
of~$\al|_S$.

Consider the map $\vf = (\al|_S \circ (a  \x b))^{-1}
\: U = \al|_S(A' \x B') \to I^k \x I^{n-k}$ with coordinate functions
$\vf = (u,v)$ taking values in $I^k \x I^{n-k}$. For any constant
vector $v = \vec{c} \in I^{n-k}$, there is a slice of the form
$\al|_S(A' \x \{y\}) \subseteq G\.y$, for $y \in B'$, which is
contained in a single orbit. Because of the way the set $B'$ was
chosen, the nonempty intersection of an orbit with~$U$ must be a
single slice. Therefore, the chart $(U,\vf)$ satisfies the required
adaptedness to the action of~$G$.
\end{proof}

The next result is known as the \textit{Quotient Manifold Theorem}.
There are several proofs in the literature. The one presented here
follows parts of \cite[Thm.~5.119]{LeeJeffrey},
\cite[Thm.~21.10]{Lee}, and \cite[Thm.~1.11.4]{DuistermaatKolk}.

\begin{theorem} 
\label{thm:QM}
Let $X$ be a $n$-dimensional smooth manifold and $G$ a $k$-dimensional
Lie group with a smooth, free and proper action over $X$. Then $X/G$,
with the quotient topology, carries a unique smooth structure, with
$\dim X/G = \dim X - \dim G$, for which the canonical projection
$\pi \: X \to X/G$ is a submersion.
\end{theorem}

\begin{proof}
We first deal with the uniqueness. If we assume there are two smooth
structures $\sA$ and $\sA'$ on $X/G$ such that the respective
canonical projections $\pi$ and $\pi'$ are submersions, then
$\pi = 1 \circ \pi'$, where $1$ is the identity map
$(X/G, \sA') \to (X/G, \sA)$. For any $x \in X$, we can find local
charts $(U,\phi)$ centered at~$x$ and $(V,\psi)$ centered at $\pi(x)$
such that $\psi \circ \pi \circ \phi^{-1}$ has the form
$\bR^k \x \bR^{n-k} \ni (u,v) \mapsto v$. The maps $1$ and
$1 \circ \pi \circ \phi^{-1} \circ \iota_0 \circ \psi$ coincide on
an open set, possibly by shrinking $U$ and~$V$, showing that $1$ is
smooth. Here $\iota_0(v) = (0,v) \: \bR^{n-k} \to \bR^k \x \bR^{n-k}$.
The same result holds similarly for its inverse, making $1$ a
diffeomorphism. We conclude that the smooth structure is unique.

From Lemma~\ref{lm:proper-action-characterization}, $X/G$ is a
Hausdorff space. It is also second countable: if $\{U_i\}$ is a
countable basis for the topology of~$X$, then $\{\pi(U_i)\}$ is a
countable basis for that of~$X/G$.

Next we show that $X/G$ is locally Euclidean. Let $z \in X/G$ and
choose $x$ with $\pi(x) = z$ and $(U,\vf)$ an action-adapted chart
centered at $x$ as in Lemma~\ref{lm:adapted-charts}, with
$\vf(U) = A \x B$ and coordinate functions
$(x^1,\dots,x^k,y^{k+1},\dots,y^n)$. We further assume that $A$ and
$B$ are open cubes in $\bR^k$ and $\bR^{n-k}$, respectively. Let
$V = \pi(U)$ and consider $W = \{x^1 =\cdots= x^k = 0\} \subseteq U$.
Since $(U,\vf)$ is an adapted chart, $\pi|_W \: W \to V$ is a
bijection. Even further, if $W' \subseteq W$ is an open subset, then
$\pi(W') = \pi(\set{(x,y) : (0,y)\in W'})$ is also open in~$X/G$.
Thus, $\pi|_W \: W \to V$ is a homeomorphism. Consider
$\sg = (\pi|_W)^{-1} \: V \to W$. Once more, because $(U,\vf)$ is
adapted, there is a well-defined map
$\bt \: V \to B : G\.(x,y) \mapsto y$. Actually,
$\bt = \pr_2 \circ \vf \circ \sg$, where $\pr_2 \: A \x B \to B$ is
the second projection. Since $\sg$ and $\pr_2 \circ \vf \: W \to B$
are homeomorphisms, so too is~$\bt$.

To finish the proof, we must exhibit a smooth structure on $X/G$ for
which $\pi$ is a submersion. We consider the atlas $\{(V,\bt)\}$ as
constructed above; since $\pi(x,y) = y$ for such a chart, it follows
that $\pi$ is a submersion.

Let $(U,\vf)$ and $(U',\vf')$ be two adapted charts for~$X$ and
$(V,\bt)$, $(V',\bt')$ the corresponding charts for~$X/G$. First, let
us assume both adapted charts are centered at the same point of~$X$,
with respective coordinates $(\vec{x},\vec{y})$ and
$(\vec{x}',\vec{y}')$. Because these charts are adapted, two points
of~$X$ with the same $\vec{y}$-coordinates lie in the same orbit and
so must have the same $\vec{y}'$-coordinates. In this way, the
transition maps $\bt' \circ \bt^{-1}$ can be written as
$(\vec{x}',\vec{y}') = (f(\vec{x},\vec{y}), h(\vec{y}))$, where $f,h$
are smooth maps defined on neighborhoods of the origin. Therefore,
$\bt' \circ \bt^{-1}$ is smooth.

For the general case, consider two adapted charts $(U,\vf)$ and
$(U',\vf')$ centered at $x$ and~$x'$, respectively, with
$\pi(x) = \pi(x')$, i.e., they belong to the same orbit: $g\.x = x'$
for some $g \in G$. The diffeomorphism $\al^g \: X \to X$ sends orbits
to orbits, so $\bar\vf = \vf' \circ \al^g$ is another adapted chart
centered at~$x$. With $\bar\sg = (\al^g)^{-1} \circ \sg'$ (the local
section for~$\vf'$) we obtain
$$
\bar\bt = \pr_2 \circ \bar\vf \circ \bar\sg 
= \pi_2 \circ \vf' \circ \al^g \circ (\al^g)^{-1} \circ \sg'
= \pi_2 \circ \vf' \circ \sg' = \bt'.
$$
Thus we are brought back to the situation of two adapted charts
centered at the same point of~$X$. As a result, $X/G$ is a smooth
manifold.
\end{proof}

It follows that if the action of $G$ on $X$ is free and proper, then
$\dim G \leq \dim X$.

\medskip

The next result is the converse to
Proposition~\ref{pr:principal-bundle-action-proper-free} on the
characterization of principal bundles through the properties of the
action. As observed in \cite[p.~55]{DuistermaatKolk}, ``\dots\ having
a proper and free action of $G$ on~$M$ is equivalent to saying that
$M$ is a principal fiber bundle with structure group~$G$.''

\begin{theorem} 
\label{thm:characterization}
Let $X$ be a smooth manifold and $G$ a Lie group with a smooth, free
and proper action over $X$. Then $(X,\pi,X/G,G)$ is a smooth principal
$G$-bundle: every $G\.x \in X/G$ has an open neighborhood $V$ and a
diffeomorphism $\psi \: \pi^{-1}(V) \to G \x V
: x' \mapsto \bigl( \chi(x'),\pi(x') \bigr)$ such that
$\psi(g\.x') = \bigl( \chi(x')g, \pi(x') \bigr)$ for $g \in G$.
\end{theorem}

\begin{proof}
Let $x \in X$ be fixed. We use the notation from the previous proof,
taking as $V$ the same neighborhood in $X/G$, recalling that the
sets$A$ and $B$ can be taken as open cubes and that, for any
$\vec{c} \in B$, $\vf^{-1}(A \x \vec{c})$ is the intersection of $U$
with the orbit through any point $\vf^{-1}(\vec{a},\vec{c})$. The
following diagram, where top and bottom maps are local charts, is
commutative:
\[
\xymatrix{
U \ar[r]^(.4){\vf} \ar[d]_\pi  & A\x B \ar[d]^{\pr_2} 
\\
V \ar[r]^\bt &  B. 
}
\]

Consider the smooth map $\psi^{-1} \: G \x V \to \pi^{-1}(V) 
: (g,\pi(x')) \mapsto g\.\vf^{-1}\bigl(\vec{0}, \bt(\pi(x'))\bigr)$.
For any $x \in \pi^{-1}(V)$ and any $g \in G$, there holds
$$
\pi\bigl( g\.\vf^{-1}\bigl( \vec{0}, \bt(\pi(x')) \bigr) \bigr)
= \pi \circ \vf^{-1}\bigl( \vec{0}, \bt(\pi(x')) \bigr)
= \bt^{-1} \circ \pr_2\bigl( \vec{0}, \bt(\pi(x')) \bigr)
= \pi(x') \in V.
$$
Hence, there exists $g' = g'_{g,x'} \in G$, depending on $g$ and~$x'$,
such that $g'g\.\vf^{-1}\bigl( \vec{0}, \bt(\pi(x')) \bigr) = x'$.
Thus $\psi^{-1}$ is surjective. Also, if
$g\.\vf^{-1}\bigl( \vec{0}, \bt(\pi(x')) \bigr) 
= h\.\vf^{-1}\bigl( \vec{0}, \bt(\pi(x'')) \bigr)$, then
$g'g\.x' = h'h\.x''$ for some $h'\in G$. Therefore
$\pi(x') = \pi(x'')$ and hence $g = h$, ensuring that $\psi^{-1}$ is a
smooth bijection.

To finish, we must show smoothness of 
$\psi \: \pi^{-1}(V) \to G \x V$. Taking $g = e$ one gets
$x' = g'_{e,x'}\.\vf^{-1}\bigl( \vec{0}, \bt(\pi(x')) \bigr)$, whereby
$(g'_{e,x'})^{-1}
= \tau\bigl( x', \vf^{-1}\bigl( \vec{0}, \bt(\pi(x')) \bigr) \bigr)$.
By Corollary~\ref{cor:diffeomorphic-orbits} the restriction of 
$\tau_{x'}$ is smooth, as are the inversion and
multiplication maps on~$G$. Thus,
$\psi(x') = \bigl( \tau\bigl( x', \vf^{-1}\bigl(
\vec{0}, \bt(\pi(x')) \bigr) \bigr)^{-1}, \pi(x') \bigr)$ is a smooth
map. The result follows by taking 
$\chi(x') := \tau\bigl(\vf^{-1}(\vec{0}, \bt(\pi(x')), x'\bigr)$.
\end{proof}

\begin{lemma} 
\label{lm:diffeomorphic-fibers}
In any smooth principal fiber $G$-bundle $(X,p,B)$, each fiber is
diffeomorphic to~$G$.
\end{lemma}

\begin{proof}
In Proposition~\ref{pr:existence-of-f}, it was shown that for any
principal fiber $G$-bundle $(X,p,B,G)$, there is a homeomorphism
$f \: X/G \to B : G\.x \mapsto p(x)$ making it a topological
$G$-bundle. In the setting of a principal action, Theorem~\ref{thm:QM}
says that $X/G$ is a manifold and Theorem~\ref{thm:characterization}
shows that each $G\.x \in X/G$ has an open neighborhood $V$ in $X/G$
with a diffeomorphism $\psi \: \pi^{-1}(V) \to G \x V$ for which
$f \circ \pr_2 = p \circ \psi^{-1} \: G \x V \to p(\pi^{-1}(V))$. 
Hence, for all $g \in G$,
$f(x\.G) = f \circ \pr_2 (g,G\.x) = p(\psi^{-1}(g,G\.x))$; so that $f$
is a smooth map.
\[
\xymatrix{\ar @{} [dr]
G  \x V \ar[d]_{\pr_2} \ar[r]^{\psi^{-1}}
& \pi^{-1}(V) \ar[d]^{\pi|} \ar[r]^1 & \pi^{-1}(V) \ar[d]^{p|} &
\\
V \ar[r]^1   & V \ar[r]^(.35){f} & p(\pi^{-1}(V)).
}
\]

Since $df_{G\.x} \circ (d\pr_2)_{(g,G\.x)}
= dp_{\psi^{-1}(g,G\.x)} \circ d\psi^{-1}_{(g,G\.x)}$ where
$d\psi^{-1}_{(g,G\.x)}$ is an isomorphism and
$dp_{\psi^{-1}(g,G\.x)}$ and $(d\pr_2)_{(g,G\.x)}$ are epimorphisms,
it follows that $df_{G\.x}$ is an epimorphism from $T_{G\.x} V$ to
$T_{p(x)}p(\pi^{-1}(V))$, that have the same dimension. As a result,
$f$~is a diffeomorphism, with $f(G\.x) = p^{-1}(x)$. From
Proposition~\ref{pr:principal-bundle-action-proper-free} and
Corollary~\ref{cor:diffeomorphic-orbits}, the result follows.
\end{proof}

We return now to right actions of $G$ on~$X$. For the next results, it
is enough to recall the definition of an associated fiber bundle
$((X \x F)/G,p_F,B,F)$, the $G$-structure being given by the relation
$(x,y)\.g = (x\.g,g^{-1}\.y)$, and
$p_F \: (X \x F)/G \to B : (x,y)\.G \mapsto p(x)$.

\begin{lemma} 
\label{lm:free-action-on-XxF}
If the action of $G$ on $X$ is free then the action of $G$ on $X \x F$
is free, and so the translation map
$\tau_ \x \: (X \x F) \x_G (X \x F) \to G$ is well defined.
\qed
\end{lemma}

Since 
\begin{align*}
\MoveEqLeft{(X \x F) \x_G (X \x F) 
= \set{((x,y), (x,y)\.g) : x \in X, \ y \in F, \ g \in G}}
\\
&= \set{((x,y), (x\.g,g^{-1}\.y)) : x \in X, \ y \in F, \ g \in G} 
\subset (X \x F) \x (X \x F),
\end{align*}
the map $\tau_\x$ has the factorization $(X \x F) \x_G (X  \x F) 
\stackrel{\pr_{13}}{\longrightarrow} X \x_G X \stackrel{\tau}{\to} G$,
where the $\pr_{13}$ is the restriction to $(X \x F) \x_G(X \x F)$ of
the projection from $(X \x F) \x (X \x F)$ to the first and third
factors. The following lemma is immediate.

\begin{lemma} 
\label{lm:assoc-transl-cont}
If the translation map $\tau \: X \x_G X \to G$ is continuous, then
the translation map $\tau_\x \: (X \x F) \x_G (X \x F) \to G$ is also
continuous.
\qed
\end{lemma}

By Lemma~\ref{lm:proper-action-characterization}, the presence of a
continuous, free and proper action guarantees that the translation map
is continuous, so we can state the following.

\begin{proposition} 
\label{pr:F-homeomorphic-fibers}
Let $(X,p,B)$ be a smooth principal fiber $G$-bundle with associated
fiber bundle $((X \x F)/G,p_F,B,F)$. For each $b \in B$, the fiber $F$
is homeomorphic to $p_F^{-1}(b)$.
\end{proposition}

\begin{proof}
Let $p(x_0) = b$ for some $x_0 \in X$, and let
$\rho_{x_0} \: F \to (X \x F)/G : y \mapsto (x_0,y)\.G$, which is
continuous. Since $p_F((x_0,y)\.G) = p(x_0) = b$, one can view
$\rho_{x_0} \: F \to p_F^{-1}(b)$ by restricting its range from
$(X \x F)/G$ to $p_F^{-1}(b)$.

Consider the continuous map $k_1 \: p^{-1}(b) \x F \to F
: (x,y) \mapsto \tau(x_0,x)\.y$. Clearly,
$k_1(x\.g, g^{-1}\.y) = k_1(x,y)$. By factoring $k_1$ through
the restriction of the quotient map $\pi_\x \: X \x F \to (X \x F)/G$,
we get a continuous map $k \: p_F^{-1}(b) \to F$. By construction,
$\rho_{x_0}$ and $k$ are inverse to each other.
\end{proof}

Next we must ensure that the quotient space $(X \x F)/G$ has a smooth
manifold structure. This is offered by the next lemma; and
furthermore, Theorem~\ref{thm:QM} has the consequence that
$\dim((X \x F)/G)) = \dim X + \dim F - \dim G = \dim(X/G \x F)$.

It is tempting to consider a possible diffeomorphism between
$(X \x F)/G$ and $X/G \x F$. The obvious map
$\ka \: (x,y)\.G \mapsto (x\.G,y)$ is well defined if and only if
$(x\.G,y) = \ka((x,y)\.G)) = \ka((x\.g, g^{-1}\.y)\.G)
= (x\.G, g^{-1}\.y)$, if and only if $g\.y = y$ for all~$y$ and~$g$;
which means that the action of $G$ on~$F$ is trivial.
Nevertheless, since $(X,p,B)$ is a smooth $G$-bundle, there is a
diffeomorphism $f \: B \to X/G$ and so 
$f \circ p_F \: (X \x F)/G \to B \to X/G$ is a smooth surjective map.

\begin{lemma} 
\label{lm:GonX-prin-GonXF-prin}
If the action of $G$ on $X$ is smooth and principal, then the action
of $G$ on $X \x F$ is principal. Thus, the quotient space
$(X \x F)/G$ has a smooth manifold structure with
$\dim(X \x F)/G = \dim(X/G \x F)$; and both
$\pi_\x \: X \x F \to (X \x F)/G$ and
$p_F \: (X \x F)/G \to B : (x,y)\.G \mapsto p(x)$ are smooth maps.
\end{lemma}

\begin{proof}
From Lemma~\ref{lm:free-action-on-XxF}, the action of $G$ on $X \x F$
is free. To see that it is proper, let $(x_n,y_n)$ be a convergent
subsequence in $X \x F$ and $(g_n)$ a sequence in $G$ such that
$(x_n\.g_n, g_n^{-1}\.y_n)$ also converges in $X \x F$. Then
$(x_n\.g_n)$ is convergent in~$X$, and from
Lemma~\ref{lm:equivalent-proper-actions} there is a convergent
subsequence of $(g_n)$, making the action on $X \x F$ proper as well.
The result follows from Theorem~\ref{thm:QM}.
\end{proof}

Recall that for a smooth fiber bundle $(X,p,B,G)$ with fiber a
Lie group $G$, the map $p \: E \to B$ is smooth and each point of
$B$ has a neighborhood $U$ and a diffeomorphism
$\psi \: U \x F \to p^{-1}(U)$ such that
$p \circ \psi = \pr_1 \: U \x F \to U$.

Following Proposition~\ref{pr:sets-as-bundles}, the maps
$\psi_{ij} = \psi_i^{-1} \circ \psi_j
\: (U_i \cap U_j) \x G \to (U_i \cap U_j) \x G$ are diffeomorphisms.
Thus, the map $(b,g) \mapsto (b,h)$, where 
$\sg_j(b)\.g = \sg_i(b)\.h$, is a diffeomorphism.

Lemma~\ref{lm:GonX-prin-GonXF-prin} shows that an associated bundle
$((X \x F)/G,p_F,B,F)$ to a smooth principal $G$-bundle $(X,p,B)$ is
indeed a smooth bundle. Next -- see \cite[Proposition~I, p.~198,
vol.~II]{Greubetal} -- we establish that it has a local product
structure making it a smooth fiber bundle, similarly to
Proposition~\ref{pr:proper-discont-fiber-bundle}.

\begin{theorem} 
\label{thm:associated-bundle}
If the action of $G$ on $X$ is smooth and principal, then there is a
unique smooth structure on $(X \x F)/G$ such that:
\begin{enumerate}[noitemsep, label={\textup{(\roman*)}}]
\item 
$((X \x F)/G,p_F,B,F)$ is a smooth fiber bundle;
\item 
$\pi_\x \: X \x F \to (X \x F)/G$ is a smooth fiber-preserving map,
restricting to diffeomorphisms $\pi_{\x\,x} \: F \to p_F^{-1}(p(x))$
on each fiber; and
\item 
$(X \x F,\pi_\x,(X \x F)/G,G)$ is a smooth principal bundle.
\end{enumerate}
\end{theorem}

\begin{proof}
(i):\enspace
Let $\{U_i\}$ be an open cover of $B$ and consider local sections
$\sg_i \: U_i \to X$. Since $(X,p,B)$ is a principal bundle, the
fibers coincide with the orbits: for $b \in U_i \cap U_j$ there holds
a relation $\sg_j(b) = \sg_i(b) \. \phi_{ij}(b)$, where each
$\phi_{ij} \: U_i \cap U_j \to G$ are smooth.

Define $\psi_i \: U_i \x F\to p_F^{-1}(U_i)
: (b,y) \mapsto \pi_\x(\sg_i(b),y)$. For each $b \in U_i$,
$p_F(\psi_i(b,y)) = x$ and $\psi_i$ restricts to maps
$\psi_{i,b} \: p_F^{-1}(b) \to F$. Each orbit in $p_F^{-1}(b)$
corresponds to a unique $y \in F$ such that the orbit passes through
$(\sg_i(b),y)$; $\psi_{i,b}$ is bijective and hence $\psi_i$ is also
bijective. Because
\begin{align*}
\psi_i^{-1} \circ \psi_j(b,y)
&= \psi_i^{-1}(\pi_\x(\sg_j(b), y))
= \psi_i^{-1}(\pi_\x(\sg_i(b)\.\phi_{ij}(b), y))
\\
&= \psi_i^{-1}(\pi_\x(\sg_i(b), \phi_{ij}(b)\.y))
= (b,\phi_{ij}(b)\.y),
\end{align*}
from Proposition~\ref{pr:sets-as-bundles}, with the maps
$\psi_i^{-1} \circ \psi_j$, one obtains a smooth structure on
$(X \x F)/G$ making $((X \x F)/G,p_F,B,F)$ a smooth fiber bundle with
coordinate representation $\{(U_i,\psi_i)\}$.

By Proposition~\ref{pr:sets-as-bundles}, the smooth structure one
obtains on $(X \x F)/G$ is the unique one for which $(U_i,\psi_i)$ is
a coordinate representation; but this is not yet the uniqueness
property one wants to establish.

(iii):\enspace
Consider the commutative diagram
\[
\xymatrix{\ar @{} [dr]
U_i \x G \x F 
\ar[d]_{1 \x {\,\cdot}} \ar[r]^(.5){\phi_i \x 1} 
& p^{-1}(U_i) \x F \ar[d]^{\pi_\x} 
\\
U_i \x F \ar[r]^(.5){\psi_i} & p_F^{-1}(U_i),
}
\]
with $\phi_i(b,g) = \sg_i(b)\.g$ and $1$ representing both identity
maps.

Set $V_i = p_F^{-1}(U_i)$ to have
$\pi_\x^{-1}(V_i) = (p_F \circ \pi_\x)^{-1}(U_i) = p^{-1}(U_i) \x F$.  
In this way, there are diffeomorphisms
$\chi_i \: V_i \x G \to \pi_\x^{-1}(V_i)
: (\psi_i(b,y), g) \mapsto (\phi_i(b,g), g^{-1}\.y)$, that satisfy the
relations $(\pi_\x \circ \chi_i)(w,g) = w$ and
$\chi_i(w,gh) = \chi_i(w,g)\.h$, for $w \in V_i$, $g,h \in G$. Since
this is a fiber bundle, (iii)~follows.

(ii):\enspace
From the commutation relation
$p \circ \pr_1 = p_F \circ \pi_\x \: X \x F \to B$, one can see that
$\pi_\x$ is fiber preserving. From the commutative diagram above,
one deduces that the maps $\pi_{\x,b} \: F \to p_F^{-1}(b)$, with
$b \in B$ fixed, are diffeomorphisms.

Uniqueness: From~(i), there is already a smooth structure on~$(X \x
F)/G$. Write $X_F'$ to denote the same set with a second smooth
structure. Assume that all properties in the statement remain valid on
replacing $(X \x F)/G$ by~$X_F'$. In the consequent diagram
\[
\xymatrix{\ar @{} [dr]
& X  \x F \ar[dl]_{\pi_\x}  \ar[dr]^{\pi'_\x} & 
\\
(X \x F)/G \ar[rr]^(.5){1} & & X_F'
}
\]
one wants to show that $1$ is a diffeomorphism. Trivially, $1$ is
the only map of sets making the diagram commute. From~(ii), both
$\pi_\x$ and $\pi'_\x$ are smooth. From~(iii), $d_{(x,y)} \pi_\x$ is
surjective for every $x \in P$, $y \in F$ (for this, one only needs
the fiber-bundle structure). In this way, both
$1 \: (X \x F)/G \to X'_F$ and its inverse map
$1' \: X'_F \to (X \x F)/G$ have surjective differentials at each
point. One concludes that this~$1$ is a diffeomorphism; which
completes the proof.
\end{proof}


\section{Sections and equivariant maps} 
\label{sec:equivariants}

In this section, Theorem~\ref{thm:sections-equiv-maps} (taken from
\cite{Husemoller}, see also \cite[Lemma~1.16,
p.~14]{BaumHajacMatthesSzymanski}) is studied in the smooth situation.
Remember that on $(X \x F)/G$ the action is
$(x,y)\.g = (x\.g,g^{-1}\.y)$ and there is a map
$\pi_F \: (X \x F)/G \to X/G : (x,y)\.G \mapsto \pi(x)$.%
\footnote{%
We identify a principal bundle $(X,p,B)$ with $(X,\pi,X/G)$ via the
isomorphism of principal bundles $(1,f)$.}
The following constructions are common for the three situations: a
topological group $G$ acting on a topological space $X$, a discrete
group $G$ acting on a smooth manifold $X$, and a Lie group $G$ acting
on a smooth manifold~$X$.

Before getting to it, recall Proposition~\ref{pr:local-sections}.
Mimicking that situation, if for each $x \in X$ there is an open set
$U \ni \pi(x)$ and a diffeomorphism $\psi \: U \x G \to \pi^{-1}(U)$
such that $\pr_1 = \pi \circ \psi$, we can define the smooth map 
$s_\psi \: U \to \pi^{-1}(U) : u \mapsto \psi(u,e)$ with
$\pi \circ s_\psi = 1_U$ and $x \sim s_\psi(\pi(x))$. To ensure that
the given $x \in X$ belongs to the image of a smooth section, we
require $s_\psi \cdot \tau(x,\psi(u,e))$ to be smooth, which is
justified by Corollary~\ref{cor:diffeomorphic-orbits}. On the other
hand, if $s \: U \to X$ is a local smooth section through the
point~$x$, we define the smooth map $\psi_s \: U \x G \to \pi^{-1}(U)$
by $\psi_s(u,g) := s(u)\.g$. This map enjoys all properties of the
continuous case; there only remains the smoothness of
$\psi_s^{-1} \: \pi^{-1}(U) \to U \x G$, which once again follows from
the relation $y = \psi_s(\pi(y),e) \cdot \tau(\psi_s(\pi(y),e),y)$. We
obtain the following resulting statement about existence of local
smooth sections.

\begin{proposition} 
\label{pr:local-smooth-sections}
If the action of $G$ on $X$ is smooth and free, then $(X,\pi,X/G)$ is
a principal fiber $G$-bundle if and only if there is a local smooth
section through each point of~$X$.
\qed
\end{proposition}

Let $s$ be a cross section of the associated bundle
$((X \x F)/G,\pi_F,X/G)$; consider the map $\phi_s \: X \to F$ defined
by the relation $s(\pi(x)) = (x, \phi_s(x))\.G$ for each $x \in X$; it
is well defined insofar as the action of $G$ on $X$ is free. Because
$(x, \phi_s(x))\.G = (x\.g, \phi_s(x\.g))\.G
= (x\.g, g^{-1}\.\phi_s(x))\.G$ and $s$ is a cross section, the
function $\phi_s$ satisfies the relation
$\phi_s(x\.g) = g^{-1}\.\phi_s(x) \in F$ for all $x \in X$, $g \in G$.
In other words, $\phi_s \: X \to F$ is a \textit{$G$-equivariant map}.

On the other hand, let $\phi \: X \to F$ be any equivariant map. We
associate to~it a function
$s_\phi \: X/G \to (X \x F)/G : \pi(x) \mapsto (x,\phi(x))\.G$. Since
$(x\.g,\phi(x\.g))\.G = (x\.g,g^{-1}\.\phi(x))\.G = (x,\phi(x))\.G$ in
$(X \x F)/G$, this function $s_\phi$ is well defined.

\begin{theorem} 
\label{thm:sections-equiv-maps}
Let $(X,\pi,X/G)$ be a (topological) principal $G$-bundle, and
consider an associated (topological) fiber bundle
$((X \x F)/G,\pi_F,X/G)$, where $F$ is a left $G$-space. The cross
sections $s$ of the latter bundle are in bijective correspondence with
continuous equivariant maps $\phi \: X \to F$.
\end{theorem}

\begin{proof} 
By definition, $s_\phi \circ \pi = \pi_F \circ (1 \x \phi)$ and it
follows that $s_\phi$ is continuous because $\pi$ and $\pi_F$~are.
Indeed, $O \subseteq (X \x F)/G$ is open if and only if
$\pi_{\x}^{-1}(O)$ is open in $X \x F$ and
$\pi^{-1} \circ s_\phi^{-1}(O) = (s_\phi \circ \pi)^{-1}(O)
= (\pi_{\x} \circ (1 \x \phi))^{-1}(O)
= (1 \x \phi)^{-1}\circ \pi_{\x}^{-1}(O)$ is open in $X$. Clearly, the
relation $\pi_F(s_\phi(\pi(x))) = \pi_F((x,\phi(x))\.G) = \pi(x)$
holds, and so $s_\phi$ is a cross section.
\[
\xymatrix{\ar @{} [dr]
X \ar[r]^(.4){1\x\phi} \ar[d]_{\pi}  & X \x F \ar[d]^{\pi_\x} 
\\
X/G \ar[r]_{s_\phi} &  (X \x F)/G.
}
\]

Finally, a proof of continuity of $\phi_s$ will establish the theorem,
because the maps $\phi \mapsto s_\phi$ and $s \mapsto \phi_s$ are
inverse to each other.

Let $x_0 \in X$, $y_0 = \phi_s(x_0) \in F$, $b_0 = \pi(x_0) \in X/G$,
and $s(b_0) = (x_0,y_0)\.G \in (X \x F)/G$; and let $W$ be an open
neighborhood of~$y_0$. By continuity of the action of $G$ on~$F$,
there are open neighborhoods $W'$ of~$y_0$ in $F$ and $N$ of~$e$ in
$G$ such that $N\.W' \subseteq W$. Since the translation function
$\tau \: X \x_G X \to G$ of $(X,\pi,X/G)$ is continuous,
$\tau^{-1}(N)$ is open in $X \x_G X$ with its relative topology, and
since $\tau(x_0,x_0) = e \in N$, there is an open neighbourhood $V$ of
$x_0$ in $X$ such that $\tau((V \x V) \cap (X \x_G X)) \subseteq N$.
Since $s$ is continuous, there is an open neighborhood $U$ of~$b_0$
such that $s(U) \subseteq (V \x W')/G$.

Consider $V' := \pi^{-1}(U) \cap V$. The relation
$s(U) \subseteq (V' \x W')/G$ holds, with $\pi(V') \subseteq U$. The
aim is to prove that $\phi_s(V') \subseteq W$. Let $x\in V'$ and
$b = \pi(x) \in U$; then $s(b) = (x',y')\.G$, where $x' \in V'$ and
$y' \in W'$. Since $s(\pi(x)) = (x',y')\.G = (x\.\tau(x,x'),y')\.G
= (x,\tau(x,x')\.y')\.G$, we deduce that
$\phi_s(x) = \tau(x,x')\.y' \in N\.W' \subseteq W$ and so
$\phi_s(V') \subseteq W$; the result follows.
\end{proof}

In Theorem~\ref{thm:sections-equiv-maps}, the equivariant maps
$\phi \: X \to F$ are continuous and the setting is the topological
one of a principal $G$-bundle. As has been the driving motivation of
this paper, we now wish to rewrite this result in the smooth manifold
situation. The same definitions of $s_\phi$ and~$\phi_s$ as in
Theorem~\ref{thm:sections-equiv-maps} are used. If one shows that
these two maps are smooth, the corresponding theorems (considering
first the case of a discrete group $G$) will be proved since
$\phi \mapsto s_\phi$ and $s\mapsto \phi_s$ are mutually inverse. In
the next result, the group $G$ is taken to be discrete and we work in
the context of a discrete principal $G$-bundle: see
Definition~\ref{def:free-covering-G-space}.

\begin{theorem} 
\label{thm:smooth-cross-sections-G-discrete}
Let $(X,\pi,X/G)$ be a discrete smooth principal $G$-bundle and take
an associated smooth fiber bundle $((X \x F)/G,\pi_F,X/G)$. The smooth
cross sections $s$ of the latter bundle are in bijective
correspondence with smooth equivariant maps $\phi \: X \to F$.
\end{theorem}

Let $(X,p,B)$ be a smooth principal $G$-bundle, and an associated
smooth fiber bundle $((X \x F)/G,p_F,B)$. The smooth cross sections
$s$ of the latter bundle are in bijective correspondence with smooth
equivariant maps $\phi \: X \to F$.

\begin{proof}
By Theorem~\ref{thm:DoCarmo} the projection $\pi \: X \to X/G$ is a
local diffeomorphism and the identity
$s_\phi\circ \pi = \pi_F\circ (1 \x \phi)$ shows that $s$ is smooth,
since Lemma~\ref{lm:covering-space-implies-covering-space} proves that
$\pi_F \: (X \x F)/G \to B$ is smooth. From the proof of
Theorem~\ref{thm:sections-equiv-maps}, it follows that
$\phi_s \: X \to F$ given by the relation
$s(\pi(x)) = (x,\phi_s(x))\.G$ for each $x \in X$, with $s$ a cross
section, is an equivariant map; it remains only to show that $\phi_s$
is smooth.

The keys for that are Corollaries
\ref{cor:quotient-manifold-and-tau-smooth-1}
and~\ref{cor:quotient-manifold-and-tau-smooth-2}, and
Lemma~\ref{lm:covering-space-implies-covering-space}. They state that
$X \x_G X$, $X \x G$ and $X/G$ are (locally) diffeomorphic with smooth
translation maps, and that $(X \x F)/G$ is diffeomorphic to $X \x F$.
We shall identify $X \x_G X \equiv X/G \equiv X \x G$; $\tau$ with the
projection $\pr_2 \: (x,g) \mapsto g$; and
$(X \x F)/G \equiv X \x F$. Under these identifications, $s$~is the
smooth map $X \to X \x F : x \mapsto (x,\phi_s(x))$, and the result
follows from the relation $\phi_s = \pr_2 \circ s \: X \to F$.
\end{proof}

Our very last result deals with the case of a Lie group $G$ acting
smoothly, freely and properly on a smooth manifold~$X$, with another
smooth action on a smooth manifold~$F$. See \cite[Proposition~1.7,
p.~18]{Berlineetal}.

\begin{theorem} 
\label{thm:smooth-cross-sections-G-Lie-group}
Let $(X,p,B)$ be a smooth principal $G$-bundle, and an associated
smooth fiber bundle $((X \x F)/G,p_F,B)$. The smooth cross sections
$s$ of the latter bundle are in bijective correspondence with smooth
equivariant maps $\phi \: X \to F$.
\end{theorem}

\begin{proof}
From the proof of Theorem~\ref{thm:sections-equiv-maps} we already
know that $\phi_s$ is an equivariant map and that the maps $s_\phi$
and $\phi_s$ are continuous. We must show that these functions are
smooth. Let $\vf \: U\subseteq X \to \vf(U) \subseteq \bR^n$,
$\psi \: W \subseteq F \to \psi(W) \subseteq \bR^l$,
$\al \: V \subseteq X/G \to \al(V) \subseteq \bR^{n-k}$, and
$\bt \: O \subseteq (X \x F)/G \to \bt(O) \subseteq \bR^{n+l-k}$ be
respective local charts.

From the relation $s\circ \pi = \pi_F \circ (1 \x \phi)$,
restricting to appropriate intersections, we obtain
$$
\bigl( \bt \circ s \circ \al^{-1} \bigr) 
\circ \bigl( \al \circ \pi \circ \vf^{-1} \bigr)
= \bigl( \bt \circ \pi_\x \circ (\vf \x \psi)^{-1} \bigr)
\circ \bigl( (\vf \x \psi) \circ (1 \x \phi) \circ \vf^{-1} \bigr).
$$
 
Since $\al \circ \pi \circ \vf^{-1}$ and 
$\bt \circ \pi_\x \circ (\vf \x \psi)^{-1}$ are smooth surjective maps
from $\bR^n$ to $\bR^{n-k}$ and from $\bR^{n+l}$ to $\bR^{n+l-k}$,
respectively, it follows that $\bt \circ s \circ \al^{-1}$ from
$\bR^{n-k}$ to $\bR^{n+l-k}$ is smooth if and only if
$(\vf \x \psi) \circ (1 \x \phi) \circ \vf^{-1}$ from $\bR^n$ to
$\bR^{n+l}$ is smooth, and the result follows.
\end{proof}


\section{Conclusion} 
\label{sec:finito}

The search for an appropriate noncommutative replacement for principal
bundles has been an objective of considerable significance over the
last 30~years, as is established in a rough timeline starting around
\cite{BrzezinskiMajid} until \cite{Tobolski} (to mention only two). 
This task relies on different aspects. First, a clear understanding
and characterization of what is a (commutative) principal bundle.
Second, the desire to solve problems whose solution eludes the already
known tools (see \cite{HajacMajid}, among many others). Third, the
generalization of common or similar properties with other mathematical
objects (e.g., Galois extensions as in~\cite{BrzezinskiHajac}).
Fourth, the intrinsic nature of noncommutative geometry, wherein
spaces of points get replaced by possibly noncommutative algebras of
functions over a putative nonexistent space (see for
example~\cite{Kassel}).

On the latter aspects we hope to report soon; these notes focus on the
first of those motivations. For a topological $G$-bundle, the action
is free if and only if the canonical map is injective, and the action
is continuous if and only if the canonical map is continuous. In the
presence of a free action the translation map is well defined, and its
continuity is equivalent to that of the inverse of the canonical map.
For the canonical map of a free and continuous action to be a
homeomorphism, it requires continuity of its inverse, that is to say,
the closedness of the canonical map. This is achieved by asking that
the action be proper. That request is in turn equivalent to having a
Hausdorff topology on the quotient (base) space. A principal (i.e.,
free and proper) action guarantees that each orbit of the bundle is
homeomorphic to the group. What is missing to obtain a principal
bundle is for the base space to be Hausdorff. The final conclusion of
Sections \ref{sec:trans-map-topol-spaces}
and~\ref{sec:topological-bundles} is that existence of a continuous
translation map does not characterize a topological principal
$G$-bundle: in the more general setting the base space need not be
Hausdorff.

This situation is nicely illustrated with the Milnor construction
which produces a principal action of the group on the total space,
that becomes a principal bundle insofar as the group is Hausdorff.

Proceeding from the topological to the smooth situation, in these
notes we considered as an intermediate step the case of a discrete
group acting on a smooth manifold. Thereby the quotient space is
Hausdorff and we obtain a principal bundle with a continuous
translation~map.

For the case of a Lie group acting on a smooth manifold, the
requirement of a Hausdorff topology is replaced by having a smooth
structure on the quotient space. Such is ensured by the well known
Quotient Manifold Theorem. As indicated in
Corollary~\ref{cor:diffeomorphic-orbits}, the continuity of the
translation map is reflected in the smoothness of its restriction to
each orbit.

An obligatory step, more than a detour, is our short study of
associated fiber bundles, equivariant maps and continuous or smooth
sections. By Proposition~\ref{thm:smooth-cross-sections-G-Lie-group}
and its immediate adaptation to the topological case, a smooth (or
continuous) free action produces a principal bundle if and only if
there is a local smooth (respectively, continuous) section through
each point of the total space. For an associated fiber bundle, in any
of the three situations dealt with here local cross sections are in
bijective correspondence with smooth (or continuous) equivariant maps
from the total space to the associated fiber.

\subsection*{Acknowledgments}

Many thanks to Joseph C. Várilly for several discussions. Support from
the Vicerrectoría de Investigación of the Universidad de Costa Rica is
acknowledged.




\end{document}